\documentclass[10pt,reqno]{amsart}

\usepackage{color}
\usepackage{mathrsfs}
\usepackage{enumitem}
\usepackage{graphicx}
\usepackage[normalem]{ulem}
\usepackage{amsmath, amssymb, amsfonts, dsfont, amsthm}
\usepackage[utf8]{inputenc}
\DeclareGraphicsExtensions{.pdf,.png,.jpg}
\newtheorem{thm}{Theorem}[section]
\newtheorem{defn}[thm]{Definition}

\newtheorem{lemma}[thm]{Lemma}

\theoremstyle{remark}
\newtheorem{remark}[thm]{Remark}

\def\XXint#1#2#3{{\setbox0=\hbox{$#1{#2#3}{\int}$}
\vcenter{\hbox{$#2#3$}}\kern-.5\wd0}}

\newcommand\cbrk{\text{$]$\kern-.15em$]$}}
\newcommand\opar{\text{\,\raise.2ex\hbox{${\scriptstyle
|}$}\kern-.34em$($}}
\newcommand\cpar{\text{$)$\kern-.34em\raise.2ex\hbox{${\scriptstyle |}$}}\,}

\def\<{\langle}
\def\>{\rangle}

\def\E{{\mathds E}}

\newcommand\bL{\mathbb{L}}
\newcommand\bR{\mathbb{R}}

\newcommand\bH{\mathbb{H}}
\newcommand\bZ{\mathbb{Z}}

\newcommand\bD{\mathbb{D}}
\newcommand\bS{\mathbb{S}}

\newcommand\bN{\mathbb{N}}

\newcommand\cB{\mathcal{B}}

\newcommand\cD{\mathcal{D}}
\newcommand\cF{\mathcal{F}}
\newcommand\cG{\mathcal{G}}

\newcommand\cP{\mathcal{P}}

\newcommand\cO{\mathcal{O}}

\def\R {{\mathbb R}}

\newcommand{\mysection}[1]{\section{#1}
\setcounter{equation}{0}}

\newcommand{\one}{\ensuremath{\mathds 1}} 

\newcommand{\prob}{\ensuremath{\mathds P}} 
 
\newcommand{\wraum}{\ensuremath{\left(\Omega, \mathcal{F}, \prob\right)}}

\newcommand{\pred}{\ensuremath{\mathcal{P}}}

\newcommand{\cont}{\ensuremath{\mathcal{C}}}

\newcommand{\abs}[1]{\ensuremath{\lvert #1 \rvert}}
\newcommand{\Abs}[1]{\ensuremath{\big\lvert  #1 \big\rvert}}
\newcommand{\nrm}[1]{\ensuremath{\lVert #1 \rVert}}

\newcommand{\nnrm}[2]{\ensuremath{\lVert #1 \rVert_{#2}}}		
\newcommand{\gnnrm}[2]{\ensuremath{\big\lVert #1 \big\rVert_{#2}}}

\newcommand{\nrklam}[1]{(#1)} 
\newcommand{\rklam}[1]{\left(#1\right)}					
\newcommand{\grklam}[1]{\big(#1\big)}
\newcommand{\sgrklam}[1]{\Big(#1\Big)}
\newcommand{\ssgrklam}[1]{\bigg(#1\bigg)}
\newcommand{\rrklam}[1]{\Bigg(#1\Bigg)}	
					
\newcommand{\ggklam}[1]{\big\{#1\big\}}						
\newcommand{\sggklam}[1]{\Big\{#1\Big\}}
\newcommand{\ssggklam}[1]{\bigg\{#1\bigg\}}

\newcommand{\ssgeklam}[1]{\bigg [#1 \bigg ]}					
\newcommand{\reklam}[1]{\Bigg [#1 \Bigg ]}					

\newcommand{\dl}{\mathrm d}    
\newcommand{\dx}{\mathrm d x}
\newcommand{\dr}{\mathrm d r}
\newcommand{\ds}{\mathrm d s}
\newcommand{\dt}{\mathrm d t}
\newcommand{\dy}{\mathrm d y}
\newcommand{\dz}{\mathrm d z}
\newcommand{\du}{\ensuremath{\mathrm{d}u}}
\newcommand{\dw}{\ensuremath{\mathrm{d}w}}

\newcommand{\domain}{\ensuremath{\mathcal{O}}}   

\DeclareFontFamily{U}{matha}{\hyphenchar\font45}
\DeclareFontShape{U}{matha}{m}{n}{
      <5> <6> <7> <8> <9> <10> gen * matha
      <10.95> matha10 <12> <14.4> <17.28> <20.74> <24.88> matha12
      }{}
\DeclareSymbolFont{matha}{U}{matha}{m}{n}
\DeclareFontSubstitution{U}{matha}{m}{n}
\DeclareFontFamily{U}{mathx}{\hyphenchar\font45}
\DeclareFontShape{U}{mathx}{m}{n}{
      <5> <6> <7> <8> <9> <10>
      <10.95> <12> <14.4> <17.28> <20.74> <24.88>
      mathx10
      }{}
\DeclareSymbolFont{mathx}{U}{mathx}{m}{n}
\DeclareFontSubstitution{U}{mathx}{m}{n}
\DeclareMathDelimiter{\vvvert}{0}{matha}{"7E}{mathx}{"17}

\newcommand{\wso}{\ensuremath{K}}
\newcommand{\bwso}{\ensuremath{\mathbb{\wso}}}
\newcommand{\frwso}{\ensuremath{\mathfrak{\wso}}}
\newcommand{\wsob}{\ensuremath{H}}
\newcommand{\bwsob}{\ensuremath{\mathbb{\wsob}}}

\definecolor{felix}{rgb}{0.2,0.2,1.0} 
\definecolor{petru}{rgb}{0.7,0.1,0.1} 

\newcommand{\icp}{\color{black}}

\begin{document}

\title[Stochastic heat equations on an infinite angle]
{An $L_p$-estimate for the stochastic heat equation\\
on an angular domain in $\mathbb{R}^2$}

\author{Petru A. Cioica-Licht}
\thanks{The first author has been supported by the Deutsche Forschungsgemeinschaft (DFG, grant DA~360/20-1) and partially by the Marsden Fund Council from Government funding, administered by the Royal Society of New Zealand.
The collaboration between the authors has been significantly facilitated by a travel grant of the German Academic Exchange Program (DAAD) and the National Research Foundation of Korea (NRF) within the DAAD-NRF Bilateral Scientist Exchange Program, supporting a two-month visit of the first author to the second and the third author.}
\address{Petru A. Cioica-Licht (n\'e Cioica), Department of Mathematics and Statistics, University of Otago, PO Box~56, Dunedin 9054, New Zealand}
\email{pcioica@maths.otago.ac.nz}

\author{Kyeong-Hun Kim}
\thanks{The second author was supported by Basic Science Research Program through the National Research Foundation of Korea (NRF) funded by the Ministry of Education, Science and Technology (NRF-2014R1A1A2055538)}
\address{Kyeong-Hun Kim, Department of Mathematics, Korea University, 1 Anam-Dong, Sungbuk-gu, Seoul, 136--701, Republic of Korea}
\email{kyeonghun@korea.ac.kr}

\author{Kijung Lee}
\thanks{The third author was supported by Basic Science Research Program through the National Research Foundation of Korea (NRF) funded by the 
Ministry of Education, Science and Technology (NRF-2013R1A1A2060996)}
\address{Kijung Lee, Department of Mathematics, Ajou University, Suwon, 443--749, Republic of Korea} 
\email{kijung@ajou.ac.kr}

\author{Felix Lindner}
\address{Felix Lindner, Department of Mathematics, University of Kaiserslautern, PO Box~3049, 67653 Kaiserslautern, Germany}
\email{lindner@mathematik.uni-kl.de}



\begin{abstract}
We prove a weighted $L_p$-estimate for the stochastic convolution associated to the stochastic heat equation with zero Dirichlet boundary condition on a planar angular domain $\cD_{\kappa_0}\subset\bR^2$ with angle $\kappa_0\in(0,2\pi)$.
Furthermore, we use this estimate to establish existence and uniqueness of a solution to the corresponding equation in suitable weighted $L_p$-Sobolev spaces.
In order to capture the singular behaviour of the solution and its derivatives at the vertex, we use powers of the distance to the vertex as weight functions.
The admissible range of weight parameters depends explicitly on the angle $\kappa_0$.
\end{abstract}

\maketitle


\mysection{Introduction}\label{sec:Introduction}

\noindent In this article we make a crucial step forward towards closing the gap between the (almost) fully-fledged $L_p$-theory for second order parabolic stochastic  partial differential equations (SPDEs, for short) on $\cont^1$-domains and the much less developed theory for such equations on domains with non-smooth boundary, such as Lipschitz domains. 
To this end, it is important to understand the influence and interaction of the two main sources for spatial singularities of the solution. 
On the one hand, forcing the solution of a second order parabolic SPDE on a domain $\domain\subset\bR^d$ to fulfil, e.g., a zero Dirichlet boundary condition is in general incompatible with the roughness of the noise unless the noise vanishes near the boundary. 
As a consequence, the spatial second derivatives of the solution typically blow up near the boundary, see, e.g., \cite{Fla1990,Kry1994}. 
On the other hand, the regularity of the solution can be influenced by singularities of the boundary $\partial\domain$, such as corners and edges. This influence is already known for a long time from the analysis of deterministic equations, see, e.g., \cite{Dau1988, Gri1985, Gri1992, KozMazRos1997, Naz2001, Sol2001}. 
This list is only indicative and by no means complete.

Our primary goal in this article is a better understanding of the latter, i.e., of the influence of boundary singularities on the regularity of the solutions to SPDEs. To this end, we consider the following simplified setting that is still general enough to reveal the main issues. Let $\rklam{w^k_t}$, $k\in\bN$, be a sequence of independent one-dimensional Wiener processes on a probability space $\wraum$ and fix a time horizon $T\in (0,\infty)$. We consider the stochastic heat equation 
\begin{equation}\label{eq:SHE_Intro}
\du(t,x)=(\Delta u(t,x) +f(t,x))\, \dt+ \sum_{k=1}^\infty g^k(t,x)\,\dw^k_t,\quad (t,x)\in (0,T]\times\cD,
\end{equation}
with zero Dirichlet boundary condition on a planar angular domain
\begin{equation}\label{eq:defD}
\cD=\cD_{\kappa_0}:=\big\{x\in \bR^2: x=(r\cos\vartheta,r\sin\vartheta),\; r>0,\;\vartheta\in (0,\kappa_0)\big\},
\end{equation}
with angle $\kappa_0\in (0,2\pi)$. 
In this way, we deal with  a basic but typical example of a second order parabolic SPDE that is forced to vanish on the boundary of a domain which is smooth except at one point, the vertex $x_0=0$. 
The behaviour of the solution at this isolated singularity of $\partial\cD$ is our main interest. 

Our main result, Theorem~\ref{thm stochastic part}, is a  weighted $L_p$-estimate of the corresponding stochastic convolution ($p\geq 2$). It can be summarised as follows. Let $G$ be the Green function associated with the heat equation with zero Dirichlet boundary condition on $\cD$. Then, for suitable $\ell_2$-valued functions $g=g(\omega,t,x)$, $(\omega,t,x)\in\Omega\times(0,T]\times\cD$, the stochastic convolution 
\[
w(t,x):=\sum_{k=1}^\infty\int_0^t\int_\cD G(t-s,x,y) g^k(s,y)\,\dy\,\dw^k_s,
\qquad (t,x)\in(0,T]\times\cD,
\]
fulfils the estimate
\begin{equation}\label{eq:mainEst_Intro}
\E
\int^T_0
\int_{\cD}
\Abs{\abs{x}^{-1}w(t,x)}^p\, \abs{x}^{\theta-2} \,\dx\,\dt 
\leq
N\,
\E 
\int^T_0
\int_{\cD}
\Abs{g(t,x)}^p_{\ell_2}\,\abs{x}^{\theta-2} \,\dx\,\dt,
\end{equation}
provided the right hand side is finite and the weight parameter $\theta$ satisfies 
\begin{equation}\label{eq:range:vertex}
p \left(1-\frac{\pi}{\kappa_0}\right)<\theta<p\left(1+\frac{\pi}{\kappa_0}\right).
\end{equation}
We also show how to use this estimate to obtain weighted $L_p$-Sobolev regularity of order one  for the solution to the stochastic heat equation~\eqref{eq:SHE_Intro} with zero Dirichlet boundary condition.
In order to capture the singular behaviour of the solution at the vertex, we use a weight system based on the distance $\rho_o(x):=\abs{x}$ of a point $x\in\cD$ to the origin, i.e., to the vertex of $\cD$. 
Our main estimate~\eqref{eq:mainEst_Intro} is a stochastic version of the corresponding result for the deterministic convolution proven in~\cite{Naz2001} and almost simultaneously in~\cite{Sol2001}.
Its proof is based on a Green function estimate from~\cite{Koz1991}.

It is worth mentioning that the lower bound in~\eqref{eq:range:vertex}, which corresponds to the best integrability property of the solution near the vertex, is sharp. 
This can be seen from the following example: 
Let $\beta_t$  be a one-dimensional Brownian motion defined on a  probability space $\Omega$ and consider $\displaystyle u(\omega,t,x)=r^{\frac{\pi}{\kappa_0}}\sin\ssgrklam{\frac{\pi}{\kappa_0}\vartheta}\, \beta_t(\omega)=:g(x)\,\beta_t(\omega)$, where $x=(r\cos\vartheta,r\sin\vartheta)$. 
Since $g$ is harmonic on $\cD$ and vanishes on $\partial\cD$, $u$ satisfies
\begin{equation*}
\left.
\begin{alignedat}{3}
\du 
&= 
\Delta && u \,\dt
+
g\,\mathrm{d}\beta_t \quad \text{on } \Omega\times(0,T]\times\cD,	\\
u
&=
0 && \quad \text{on } \Omega\times(0,T]\times\partial\cD,	\\
u(0)
&=
0 && \quad \text{on } \Omega\times\cD.
\end{alignedat}
\right\}	
\end{equation*}
We are interested in  the finiteness of
\[
\E\int^T_0\int_{\mathcal{D}\cap B_{\varepsilon}} \Abs{\abs{x}^{-1}u(t,x)}^p \abs{x}^{\theta-2}\,\dx\,\dt,
\]
where  $B_{\varepsilon}$ is the open ball centred at the origin with a fixed radius $\varepsilon>0$. This integral equals
\begin{equation*}
\E\int^T_0\int^{\kappa_0}_0\int^{\varepsilon}_0\left\lvert r^{-1}r^{\frac{\pi}{\kappa_0}}\sin\left(\frac{\pi}{\kappa_0}\vartheta\right)\right\rvert^p r^{\theta-2}r\,\dr\,\dl\vartheta\, \abs{\beta_t}^p\,\dt
=
N\int^{\varepsilon}_0 r^{\nrklam{\frac{\pi}{\kappa_0}-1}p+\theta-1} \,\dr,
\end{equation*}
where 
$N=N(T,\kappa_0,p)\in (0,\infty)$.
Thus, it is finite if, and only if, the first inequality in~\eqref{eq:range:vertex} holds. 

So far, the predominant part of the literature on the regularity of SPDEs focuses on  rather smooth domains such as $\cont^1$-domains. In this case, the only source for spatial singularities of the solution---except, of course, the regularity of the data of the equation---is the incompatibility between the roughness of the noise and the imposed boundary conditions already mentioned above. 
This effect can be captured very accurately by using Sobolev spaces with weights that are appropriate powers of the distance $\rho_\domain(x):=\mathrm{dist}(x,\partial\domain)$ of a point $x\in\domain$ to the boundary $\partial \domain$ of the domain $\domain\subset\bR^d$.
Typically, the solution $u$ fulfils
\begin{equation}\label{eq:wSob:bdry:fin}
\E\int_0^T
\rrklam{
\sum_{\abs{\alpha}\leq n}
\int_\domain
\Abs{\rho_{\domain}^{\abs{\alpha}-1}(x) D^\alpha u(t,x)}^{p} \rho_\domain^{\theta-d}(x)\,\dx}\,\dt
 <\infty,
\end{equation}
where the order of differentiability $n\in\bN$ depends on the regularity of the data of the equation. 
The free terms in the deterministic and in the stochastic part, i.e., $f$ and $g=(g^k)$ in \eqref{eq:SHE_Intro}, have to be in the corresponding weighted Sobolev spaces of order $n-2$ and $n-1$, respectively, with suitably chosen weights.
The most far reaching results in this direction have been achieved within the analytic approach to SPDEs by N.V.~Krylov and collaborators: first as an $L_2$-theory for smooth domains $\domain$ (see \cite{Kry1994}), then as an $L_p$-theory ($p\geq2$) for the half space (\cite{KryLot1999, KryLot1999b}), and afterwards for general smooth domains with at least $\cont^1$-continuous boundary (\cite{Kim2004, KimKry2004}).
For instance, for the stochastic heat equation \eqref{eq:SHE_Intro} with zero Dirichlet boundary condition on a bounded  $\cont^1$-domain $\domain\subset\bR^d$ instead of $\cD$, a weighted $L_p$-theory that guarantees the existence of a solution fulfilling \eqref{eq:wSob:bdry:fin} can be established for weight parameters $\theta$ satisfying 
\begin{equation}\label{eq:range:C1}
d-1< \theta < d+p-1.
\end{equation}
This range of weight parameters cannot be kept up if we abandon the smoothness assumption for the boundary of the underlying domain $\domain\subset\bR^d$. 
This is due to the bad influence of the boundary singularities on the regularity of the solution mentioned above, see, e.g., \cite[Example~2.17]{Kim2014} for a typical counterexample.
What we know so far is that if we only assume $\domain$ to admit Hardy's inequality (every bounded Lipschitz domain does the job), then we obtain a similar theory for $\theta$ fulfilling the much more restrictive condition
\[
d+p-2-\varepsilon < \theta < d+p-2+\varepsilon,
\]
with a small $\varepsilon>0$, see \cite{Kim2014}. 
Thus, there is a big gap between the weighted $L_p$-theories available for stochastic equations on smooth domains and the ones on very rough domains. 
Not much is known about the situation in between, i.e., for SPDEs  on, e.g., polygonal or polyhedral domains. 
These types of domains are not $\cont^1$ but at the same time they are far from being ``very rough''. 
To the best of our knowledge, a refined analysis of the dependence of the admissible range of parameters on the ``strength'' of the boundary singularities is yet to be done.
Our results are a first step into this direction. 
Here we concentrate on the behaviour of the solution at the singularity of the underlying domain. Therefore, our weights are, at the moment, powers of the distance to the vertex and not of the distance to the boundary.
This is why we cannot go beyond weighted Sobolev regularity of order one. 
To obtain higher order regularities, we have to incorporate the distance to the boundary into our weight system in order to capture also the singularities that appear due to the incompatibility of the noise and  the zero boundary  condition. 
To keep the length of our exposition at a reasonable level, we postpone this analysis to a forthcoming paper.

The lower bound of the admissible range~\eqref{eq:range:vertex} is not only sharp for our weight system based on the distance to the vertex. At the same time, it seems to be the best we can expect if we go back to the analogue weight system based on the distance to the boundary,  which is usually used within the analytic approach to SPDEs. 
This can be deduced from the work of Grisvard for deterministic equations on polygonal domains, see \cite{Gri1985, Gri1992, Gri1995}. 
In~\cite{Gri1985}, he considers the (deterministic) Poisson equation $\Delta u = f$ with zero Dirichlet boundary condition on a bounded polygonal domain $\domain \subset\mathbb{R}^2$. 
He shows that if $\kappa_0$ is the maximal interior angle among all vertices of the domain and if $f$ belongs to $L_p(\cO)$ for some $p>1$, then the solution $u$ belongs to the Sobolev-Slobodeckij space $W^{s}_p(\cO)$ for any $s<\min\sggklam{2,\frac2p+\frac{\pi}{\kappa_0}}$, but in general $u$ does not have $L_p$-Sobolev regularity of order  $s\geq\min\sggklam{2,\frac2p+\frac{\pi}{\kappa_0}}$. We note that the Sobolev regularity of $u$ changes continuously according to the maximal angle $\kappa_0$ of the polygon. 
For instance, if $\kappa_0$ is close to $2\pi$, the regularity of $u$ gets worse and we cannot expect $u$ to be in the space $W^{\frac2p+\frac12+\varepsilon}_p(\cO)$ for any $\varepsilon>0$. 
An analogous result holds for the deterministic heat equation, restricting the spatial $L_p$-Sobolev regularity of the solution accordingly, see, e.g., \cite[Section~5.2]{Gri1992} and \cite[Section~8]{Gri1995}. 
For $p=2$, it has recently been extended to the case of a semilinear stochastic heat equation in \cite{Lin2014}, where a framework of Sobolev spaces without weights and with possibly negative orders of smoothness in time has been used. 
In order to relate these results to our problem, we note that if $p\geq2$ and the inner integral $\sum\int_\domain\ldots\dx$ in \eqref{eq:wSob:bdry:fin} with $d=2$ and $n\geq2$ is finite, then $u(t,\cdot)$ belongs to $W^{1+\frac{2-\theta}p}_p(\cO)$. This follows from the embedding result in \cite[Lemma~6.6]{CioKimLee+2013}, see also~\cite[Proposition~4.1]{Cio2013}. 
Therefore, $\theta$ has to fulfil 
\begin{equation}\label{eq:bound_Grisvard}
1+\frac{2-\theta}{p}<\frac2p+\frac{\pi}{\kappa_0} \quad\quad\textrm{or}\quad\quad  p\left(1-\frac{\pi}{\kappa_0}\right)<\theta,
\end{equation}
which is precisely the lower bound in~\eqref{eq:range:vertex}.
We observe that, if $\kappa_0<\pi$, i.e., if the underlying domain is convex, then this is not a further restriction to the range \eqref{eq:range:C1} of $\theta$ for $\cont^1$ domains. 

\smallskip

This paper is organized as follows. In Section~\ref{sec:Estimate} we prove our main estimate~\eqref{eq:mainEst_Intro} for the stochastic convolution (Theorem~\ref{thm stochastic part}). In Section~\ref{sec:WSobReg} we use this estimate to develop a refined weighted $L_p$-Sobolev theory for Equation~\eqref{eq:SHE_Intro}. Before we start, we fix some notation.

\smallskip

\noindent\textbf{Notation.}
Let $(X,\mathcal{A},\nu)$ be a $\sigma$-finite measure space and $\mathcal{G}\subseteq\mathcal{A}$ be a sub-$\sigma$-algebra of $\mathcal{A}$. We write $\mathcal{G}^\nu$ for the completion of $\mathcal{G}$ with respect to $(\mathcal{A},\nu)$, i.e., $\mathcal{G}^\nu$ consists of all $A\subseteq X$ such that there exists a set $B\in\mathcal{G}$ with $\one_A=\one_B$ $\nu$-almost everywhere.
Let $(E,\nnrm{\cdot}{E})$ be a Banach space.
For $p\in[1,\infty)$, we denote by $L_p(X,\mathcal{G},\nu;E)$ the space of all (equivalence classes of) $\mathcal{G}^\nu/\cB(E)$-measurable functions $f\colon X\to E$, such that
\[
\nnrm{g}{L_p(E)}:=\ssgrklam{\int_X \nnrm{f}{E}^p\,\mathrm{d}\nu}^{1/p}<\infty;
\]
$\cB(E)$ denotes the Borel $\sigma$-algebra generated by the standard topology on $E$. 
For the special case that $\nu=\sum_{j\in\bN}\delta_j$ is the counting measure on $\bN$, we write $\ell_p:=L_p(\bN,2^{\bN},\sum_{j\in\bN}\delta_j;\bR)$ for the space of $p$-summable real-valued sequences.
If $p=\infty$, then $L_\infty(X,\mathcal{A},\nu;E)$ is the usual space of all (equivalence classes of) $\mathcal{G}^\nu/\cB(E)$-measurable and essentially bounded functions $f\colon X\to E$, the norm being given by the essential supremum.
Let $\domain$ be an arbitrary domain in $\bR^d$, $d\geq 2$. If nothing else  is stated, all (generalized) scalar-valued functions are meant to be real-valued.
If $\nu$ is a measure on $(\domain,\cB(\domain))$ with density $f\colon\domain\to \bR$ with respect to the corresponding Lebesgue measure, then we write $\nu=f\,\dx$.
For $\alpha\in\bN_0^d$, we write $D^\alpha f$ for the $\alpha$-th generalized derivative of a generalized function $f$.  
As usual, $\Delta f :=\sum_{i=1}^d D^{2\mathrm{e}_i}f$, where $\mathrm{e}_i$ is the $i$-th unit vector in $\bR^d$ for  $i=1,\ldots,d$.
We also use the notation $f_{x^{i}}:=D^{\mathrm{e}_i}$ and $f_x:=(f_{x^1},\ldots,f_{x^d})$, as well as $\nrm{f_x}:=\sum_{i=1}^d\nrm{f_{x^{i}}}$, if  $\nrm{\cdot}$ is a suitable norm.  
For $p\in(1,\infty)$ and $s\geq 0$, we write $W^s_p(\domain)$ for the classical $L_p$-Sobolev-Slobodeckij space of order $s$, cf., e.g., \cite[Section~2.3.1]{Cio2013}.
The class of infinitely differentiable functions with compact support in $\domain$ is abbreviated by $\cont^\infty_0(\domain)$ and $\mathring{W}^s_p(\domain)$ is short for its closure in $W^s_p(\domain)$.
The application of a generalized function $f$ to a test function $\varphi\in\cont^\infty_0(\domain)$ is denoted by $(f,\varphi):=f(\varphi)$.
Throughout the paper, the letter $N$ is used to denote a finite positive constant that may differ from one appearance to another, even in the same chain of inequalities.


\mysection{Main estimate}\label{sec:Estimate}

\noindent The following setting is used throughout this paper: Let $(\Omega,\cF,\prob)$ be a complete probability space, and $\nrklam{\cF_{t}}_{t\geq0}$ be an increasing filtration of $\sigma$-fields $\cF_{t}\subset\cF$, each of which contains all $(\cF,\prob)$-null sets;
$\E$ denotes the expectation operator. 
We assume that on $\Omega$ we are given a family $(w_t^k)_{t\geq0}$, $k\in\bN$, of independent one-dimensional Wiener processes relative to $\nrklam{\cF_{t}}_{t\geq0}$. We fix $T\in(0,\infty)$ and denote by $\cP_T$ the predictable $\sigma$-field on $\Omega_T:=\Omega\times (0,T]$ generated by $\nrklam{\cF_{t}}_{t\geq0}$. 
Finally, we fix an arbitrary  angle $\kappa_0\in (0,2\pi)$ and consider the domain $\cD=\cD_{\kappa_0}\subset\bR^2$ given by~\eqref{eq:defD}.

In this section, we prove for suitable $\ell_2$-valued functions $g=(g^k)_{k\in\bN}$ on $\Omega_T\times\cD$ an $L_p$-estimate for the stochastic convolution
\begin{align}\label{eq:def:conv:stoch}
w(t,x)&:=\sum_{k=1}^\infty\int^t_0\int_{\mathcal{D}}G(t-s,x,y)g^k(s,y)\,\dy\, \dw_s^k, \quad (t,x)\in  (0,T]\times\cD.
\end{align}
Here, $G(t,x,y)=G_{\kappa_0}(t,x,y)$ is the Green function for the heat equation on $\cD=\cD_{\kappa_0}$ with zero Dirichlet boundary condition, defined for every $y\in\cD$ as the solution  to the problem
\begin{align*}
&\frac{\partial G(t,x,y)}{\partial t}-\Delta_x G(t,x,y)=\delta_{(0,y)}(t,x)\quad\text{in }\bR\times\cD,\\
&G(t,x,y)=0\quad\text{for }t\in\bR,\,x\in\partial\cD\setminus\{0\},\qquad G(t,x,y)=0\quad\text{for }t<0,
\end{align*}
cf., e.g., \cite[Section~1]{KozRos2012b}. The first equality above is understood in the sense of distributions. 
Given that $g$ is predictable (i.e., $\cP_T\otimes\cB(\cD)$-measurable) and satisfies a suitable integrability condition, 
the infinite sum of stochastic integrals in \eqref{eq:def:conv:stoch} is well-defined for all $t\in(0,T]$ and almost all $x\in\cD$, and we can always find a (pseudo-)predictable version of $w$, see Remark~\ref{rem:existence_stoch_conv} below for details.

\smallskip

These conventions at hand, we can state our main result.

\begin{thm}\label{thm stochastic part}
Let $p\in[2,\infty)$ and let $\theta$ satisfy \eqref{eq:range:vertex}, i.e., assume that
\begin{equation*}
p \left(1-\frac{\pi}{\kappa_0}\right)<\theta<p\left(1+\frac{\pi}{\kappa_0}\right).
\end{equation*}
Then   there exists a constant $N\in(0,\infty)$, depending only on  $p$ and $\theta$ (and not on~$T$), such that 
\begin{equation}\label{result stochastic part}
\E
\int^T_0
\int_{\cD}
\Abs{\abs{x}^{-1}w(t,x)}^p\, \abs{x}^{\theta-2} \,\dx\,\dt 
\leq
N\,
\E 
\int^T_0
\int_{\cD}
\Abs{g(t,x)}^p_{\ell_2}\,\abs{x}^{\theta-2} \,\dx\,\dt,
\end{equation}
for any $\cP_T\otimes\cB(\cD)$-measurable $\ell_2$-valued function $g=g(\omega,t,x)$ with the right hand side of \eqref{result stochastic part} finite.
\end{thm}

\begin{remark}
As mentioned in the introduction, Theorem~\ref{thm stochastic part} can be considered as a stochastic version of the corresponding deterministic result in \cite{Naz2001}, see also~\cite{Sol2001}, which states that for $p>1$ and $\mu$ satisfying
\begin{equation}\label{condition2}
2\left(1-\frac{1}{p}\right)-\frac\pi{\kappa_0}<\mu<2\left(1-\frac{1}{p}\right)+\frac\pi{\kappa_0},
\end{equation}
the integral operator $\mathcal{G}_1$ with kernel
\[
\frac{\abs{x}^{\mu-2}}{\abs{y}^{\mu\phantom{-2}}}{\one}_{t-s>0} \;G_{\kappa_0}(t-s,x,y)
\]
is bounded in $L_p((0,T)\times\cD)$. In other words,
the deterministic convolution
\begin{equation}\label{eq:def:conv:det}
v(t,x):=\int^t_0\int_{\mathcal{D}}G(t-s,x,y)f(s,y)\,\dy\,\ds
\end{equation}
satisfies the estimate
\begin{equation}\label{result deterministic part}
\int^T_0
\int_\cD
\Abs{\abs{x}^{-1}v(t,x)}^p \abs{x}^{\theta-2} \,\dx\,\dt 
\leq N 
\int^T_0
\int_\cD
\Abs{\abs{x}^1f(t,x)}^p \abs{x}^{\theta-2}\,\dx\,\dt, 
\end{equation}
whenever the right hand side of \eqref{result deterministic part} is finite and $\theta:=p(\mu-1)+2$ fulfils~\eqref{eq:range:vertex}. 
We refer to \cite[Theorem~1.1 and Lemma~2.1]{Naz2001} for details, compare also \cite[Theorem~1.2]{Sol2001}.
\end{remark}

The proof of Theorem~\ref{thm stochastic part} is based on the following Green function estimate for $G=G_{\kappa_0}$ taken from \cite{Koz1991}, compare also \cite{Naz2001}.
We use the notation 
\[
R_{x,t}:=\frac{|x|}{|x|+|y|+\sqrt{t}} \quad\text{and}\quad  R_{y,t}:=\frac{|y|}{|x|+|y|+\sqrt{t}}.
\]
\begin{thm}[\cite{Koz1991}]\label{thm:green:estimate}
For any  $0<\lambda<\displaystyle\frac{\pi}{\kappa_0}$, there exist $\sigma>0$ and $N>0$ such that
\begin{equation}\label{Green}
\Abs{G_{\kappa_0}(t,x,y)}
\leq 
\frac{N}{t}
e^{-\sigma\frac{\abs{x-y}^2}{t}} R^{\lambda}_{x,t}\,R^{\lambda}_{y,t}
\end{equation}
for all $t>0$.
\end{thm}

Using this estimate, we are able to prove our main result in the following way.

\begin{proof}[Proof of Theorem \ref{thm stochastic part}] 
Fix $\theta$ as in~\eqref{eq:range:vertex} and put $\mu:=1+(\theta-2)/p$, so that $\mu$ satisfies~\eqref{condition2}.
Set $\lambda:=\pi/\kappa_0-\varepsilon$ with $\varepsilon>0$ small enough so that
\begin{equation}\label{condition_15.12.2015_1}
2\left(1-\frac{1}{p}\right)-\lambda<\mu<2\left(1-\frac{1}{p}\right)+\lambda.
\end{equation}
Without loss of generality, we assume that the estimate \eqref{Green} is fulfilled with $\sigma=1$ and $N=1$. This is not a restriction,  since different values of $\sigma=\sigma(\lambda)>0$ and $N=N(\lambda)>0$ in \eqref{Green} only result in a possibly different constant $N$ in \eqref{result stochastic part}.  
For the sake of clarity, we split our proof into three steps.

\smallskip

\noindent\textbf{Step 1.}
We have
\begin{equation*}
\abs{x}^{\mu-2}w(t,x)
=
\sum_{k=1}^\infty \int^{t}_0\int_\cD \abs{x}^{\mu-2}\frac{G(t-s, x,y)}{\abs{y}^{\mu-1}}\,\abs{y}^{\mu-1} g^k(s,y)\,\dy\,\dw^k_s.
\end{equation*}
Let us denote $h(s,y):=\abs{y}^{\mu-1}g(s,y)$, so that assuming that the right hand side of~\eqref{result stochastic part} is finite means that $\E\nnrm{h}{L_p((0,T)\times \mathcal{D};\ell_2)}^p<\infty$.
If we set 
\begin{equation}\label{operator  stochastic}
\mathcal{G}_2h(t,x)
:=
\sum_{k=1}^\infty\int^{t}_0\int_\cD  \abs{x}^{\mu-2}\frac{G(t-s, x,y)}{\abs{y}^{\mu-1}}\,h^k(s,y)\,\dy\,\dw^k_s,
\end{equation}
then the Burkholder-Davis-Gundy inequality and the triangle inequality for $\ell_2$-valued integrals, yield 
\begin{align*}
\E
\Abs{\abs{x}^{\mu-2}w(t,x)}^p
&=
\E
\Abs{\mathcal{G}_2h(t,x)}^p\\
&\leq N\, 
\E\reklam{\int_0^t \sum_{k=1}^\infty \ssgrklam{
\int_\cD \abs{x}^{\mu-2} \frac{G(t-s, x,y)}{\abs{y}^{\mu-1}}\,h^k(s,y)\,\dy}^2\ds}^{p/2}\\
&\leq N\,
\E\reklam{ \int^t_0 \ssgrklam{\int_{\cD}  \abs{x}^{\mu-2} \frac{\Abs{G(t-s, x,y)}}{\abs{y}^{\mu-1}}\,\abs{h(s,y)}_{\ell_2} \,\dy}^2 \ds}^{p/2},
\end{align*}
where $N$ depends only on $p$. The first inequality can be justified, e.g., by considering $\cG_2h(t,x)$ as a real-valued stochastic integral w.r.t.\ a cylindrical Wiener process on $\ell_2$ and applying the Burkholder-Davis-Gundy inequality as stated in \cite[Theorem~4.36]{DaPZab2014}; if $p=2$ it is enough to apply It\^{o}'s isometry. 
We choose $\alpha$ and $\beta$ satisfying 
\begin{equation}\label{condition 2015.06.10_2}
0<\alpha<\mu+\lambda-\frac{2}{p'}\quad\text{and}\quad 0<\beta<-\mu+\lambda+\frac{2}{p'},
\end{equation}
where $p':=p/(p-1)$ is the dual exponent of $p$, i.e., $1/p+1/p'=1$. 
Then, by Hölder's inequality for the integral on $\mathcal{D}$ and the Green function estimate~\eqref{Green}, we have
\begin{align*}
\E
\Abs{\mathcal{G}_2h(t,x)}^p
&\leq
\E\reklam{ 
\int^t_0 \ssgrklam{\int_{\mathcal{D}} e^{-\frac{|x-y|^2}{t-s\phantom{n}}}
\abs{h(s,y)}^p_{\ell_2}
\,R^{\alpha p }_{x,t-s}\,R^{\beta p }_{y,t-s}\frac{1}{t-s}\,\dy }^{2/p} \\
&\phantom{\leq
\E n}\times
\ssgrklam{\int_{\cD} e^{-\frac{\abs{x-y}^2}{t-s\phantom{n}}}\frac{\abs{x}^{(\mu-2)p'}}{\abs{y}^{(\mu-1)p' }}\,R^{(\lambda-\alpha)p' }_{x,t-s}\,R^{(\lambda-\beta)p' }_{y,t-s}\frac{1}{t-s}\,\dy}^{2/p'}\ds}^{p/2}\\
&=: 
\E
\ssgeklam{
\int^t_0
I^{2/p}_{1,p}(t,s,x) \times I^{2/p'}_{2,p}(t,s,x)\,\ds}^{p/2}, 
\end{align*}
where
\begin{align*}
I_{1,p}(t,s,x)&
:=
\int_\cD e^{-\frac{\abs{x-y}^2}{t-s\phantom{n}}}\abs{h(s,y)}^p_{\ell_2}\,R^{\alpha p }_{x,t-s}\,R^{\beta p }_{y,t-s}\frac{1}{t-s}\,\dy, \\
I_{2,p}(t,s,x)&
:=
\int_\cD e^{-\frac{\abs{x-y}^2}{t-s\phantom{n}}}\frac{\abs{x}^{(\mu-2)p'}}{\abs{y}^{(\mu-1)p' }}\,R^{(\lambda-\alpha)p' }_{x,t-s}\,R^{(\lambda-\beta)p' }_{y,t-s}\frac{1}{t-s}\,\dy.
\end{align*}
Note that, since only $h$ depends on $\omega\in\Omega$, so does $I_{1,p}$, whereas $I_{2,p}$ is purely deterministic. 
By another application of Hölder's inequality to the integral on $(0,t)$ with  H\"older conjugates $p/2\geq 1$ and $p/(p-2)$, we obtain
\begin{equation}\label{estimate 2015.12.15_1}
\E
\Abs{\mathcal{G}_2h(t,x)}^p
\leq
\E 
\ssgeklam{\int^t_0 I_1(t,s,x)\,\ds} \cdot 
\gnnrm{I_{2,p}(t,\cdot,x)}{L_{\frac{2(p-1)}{p-2}}(0,t;\bR)}^{p-1},
\end{equation}
where we use the common convention $1/0:=\infty$ for $p=2$. 

\smallskip

\noindent\textbf{Step 2.} 
We estimate the second factor on the right hand side of \eqref{estimate 2015.12.15_1} and show that
\begin{equation}\label{condition 2015.12.14_2}
\gnnrm{I_{2,p}(t,\cdot,x)}{L_{\frac{2(p-1)}{p-2}}(0,t;\bR)}^{p-1}
\leq N\,\abs{x}^{-2},
\end{equation}
with a constant $N\in(0,\infty)$ that depends only on $p$, $\lambda$, $\mu$, and $\alpha$, $\beta$ from \eqref{condition 2015.06.10_2}, but not on $t$ or $x$.
We treat the cases $p=2$ and $p>2$ separately.

\smallskip

\noindent\emph{The case $p=2$.} 
Using  the fact that ${\mathcal{D}}\subset\mathbb{R}^2$ and the change of variable $y=x-z\sqrt{t-s}$ between $y$ and $z$, we have
\begin{align*}
I_{2,2}(t,s,x)
&\leq 
\int_{\mathbb{R}^2} e^{-\abs{z}^2}\frac{\abs{x}^{2(\mu-2)}}
{\abs{x-z\sqrt{t-s}}^{2(\mu-1) }}
\cdot\frac{\abs{x}^{2(\lambda-\alpha)}}{(\abs{x}+\abs{x-z\sqrt{t-s}}+\sqrt{t-s})^{2(\lambda-\alpha)}}\\
&\phantom{\int_{\mathbb{R}^2} e^{-\abs{z}^2} \frac{\abs{x}^{2(\mu-2)}}
{\abs{x-z\sqrt{t-s}}^{2(\mu-1) }}}
\times\frac{\abs{x-z\sqrt{t-s}}^{2(\lambda-\beta)}}{(\abs{x}+\abs{x-z\sqrt{t-s}}+\sqrt{t-s})^{2(\lambda-\beta)}}\,\dz\\
&=\abs{x}^{2(\lambda-\alpha+\mu-2)}
\int_{\bR^2} e^{-\abs{z}^2}\frac{ \abs{x-z\sqrt{t-s}}^{2(\lambda-\beta-\mu+1)}}{(\abs{x}+\abs{x-z\sqrt{t-s}}+\sqrt{t-s})^{2(2\lambda-\alpha-\beta)}}\,\dz.
\end{align*}
Thus, setting $a:=2(\lambda-\alpha+\mu)$ and $b:=2(\lambda-\beta-\mu)$ and noting that $a>2$ and $b>-2$ by  \eqref{condition 2015.06.10_2} with $p=2$, we obtain
\begin{align*}
I_{2,2}(t,s,x)
&\leq
\abs{x}^{a-4}
\int_{\mathbb{R}^2} 
e^{-\abs{z}^2}\frac{ \abs{x-z\sqrt{t-s}}^{b+2}}{(\abs{x}+\abs{x-z\sqrt{t-s}}+\sqrt{t-s})^{a+b}}\,\dz\\
&=
\abs{x}^{-2}\int_{\bR^2} e^{-\abs{z}^2}\frac{ \abs{x}^{a-2} \abs{x-z\sqrt{t-s}}^{2}}{(\abs{x}+\abs{x-z\sqrt{t-s}}+\sqrt{t-s})^{(a-2)+2}}\\
&\phantom{\abs{x}^{-2}\int_{\bR^2} e^{-\abs{z}^2}}
\times
\frac{ \abs{x-z\sqrt{t-s}}^{b}}{(\abs{x}+\abs{x-z\sqrt{t-s}}+\sqrt{t-s})^{b}}\,\dz\\
&\leq 
\abs{x}^{-2}
\int_{\bR^2} e^{-\abs{z}^2}\frac{ \abs{x-z\sqrt{ t-s}}^{b}}{(\abs{x}+\abs{x-z\sqrt{t-s}}+\sqrt{ t-s})^{b}}\,\dz,
\end{align*}
since $\displaystyle\frac{ \abs{x}^{a-2} \abs{x-z\sqrt{t-s}}^{2}}{(\abs{x}+\abs{x-z\sqrt{t-s}}+\sqrt{t-s})^{(a-2)+2}}\leq 1$.
Therefore, in order to prove~\eqref{condition 2015.12.14_2} for $p=2$, it is enough to show that
\begin{equation}\label{homo 2015.06.09_1}
\sup_{c>0,\, x\in \bR^2}
\int_{\bR^2} e^{-\abs{z}^2}\frac{ \abs{x-cz}^{b}}{(\abs{x}+\abs{x-cz}+c)^{b}}\,\dz 
\leq 
N<\infty 
\end{equation}
with a proper constant $N$ which depends only on $b$. Since this is obvious for $b\geq0$, we only consider  the case $b\in (-2,0)$ and prove the following equivalent statement for this case: For arbitrary $\beta'\in (0,2)$, there exists a finite constant $N$ that might depend on $\beta'$ but not on $x\in\bR^2$, so that 
\begin{equation}\label{homo 2015.06.09_3}
\int_{\bR^2}e^{-\abs{z}^2}\frac{(\abs{x}+\abs{x-z}+1)^{\beta'}}{ \abs{x-z}^{\beta'}}\,\dz = 
\int_{\bR^2}e^{-\abs{z-x}^2}\frac{(\abs{x}+\abs{z}+1)^{\beta'}}{ \abs{z}^{\beta'}}\,\dz\leq N.
\end{equation}
We first confine to the case $\abs{x}\leq 1$ and split the integral into two parts:
\begin{align*}
\int_{\bR^2} 
e^{-\abs{z-x}^2}\frac{(\abs{x}+\abs{z}+1)^{\beta'}}{ \abs{z}^{\beta'}}\,\dz
&=
\rrklam{\int_{\abs{z}\leq 1} + \int_{\abs{z}\geq 1}} e^{-\abs{z-x}^2}\frac{(\abs{x}+\abs{z}+1)^{\beta'}}{ \abs{z}^{\beta'}}\,\dz 
\\
&=:  I_{2,2}^{(1)}+I_{2,2}^{(2)}.
\end{align*}
Since $\beta'\in(0,2)$,
\begin{equation*}
I_{2,2}^{(1)}
\leq 
\int_{\abs{z}\leq 1} e^{-\abs{z-x}^2}\frac{3^{\beta'}}{ \abs{z}^{\beta'}}\,\dz
\leq
3^{\beta'} N \int^1_0 \frac{1}{r^{\beta'}}r\, \dr=N<\infty,
\end{equation*}
where the last constant $N$ depends only on $\beta'$. Also, we can use the fact that $\frac{\abs{x}+\abs{z}+1}{ \abs{z}}\leq\frac{\abs{z}+2}{\abs{z}}\leq 3$  for $\abs{z}\geq 1$ and $\abs{x}\leq 1$, to obtain
\begin{equation*}
I_{2,2}^{(2)}
\leq 
\int_{\abs{z}\geq 1}e^{-\abs{z-x}^2} 3^{\beta'}\,\dz
\leq
 3^{\beta'}\int_{\bR^2}e^{-|z|^2}\,\dz =N <\infty;
\end{equation*}
again the last $N$ depends only on $\beta'$. Thus, the case $\abs{x}\leq 1$ is proved. Next, consider $\abs{x}\ge 1$. In this case,  we use the decomposition
\begin{align*}
&\int_{\bR^2}e^{-\abs{z-x}^2}\frac{(\abs{x}+\abs{z}+1)^{\beta'}}{ \abs{z}^{\beta'}}\,\dz\\
&\qquad\quad
\leq \rrklam{\int_{\abs{z} \leq \frac{\abs{x}}{2}}  +  \int_{|z-x|\le \frac{|x|}{2}} + \int_{|z-x|\ge\frac{|x|}{2},|z|\ge\frac{|x|}{2}}} e^{-\abs{z-x}^2}\frac{(\abs{x}+\abs{z}+1)^{\beta'}}{ \abs{z}^{\beta'}}\,\dz\\
&\qquad\quad
=:  I_{2,2}^{(3)}+I_{2,2}^{(4)}+I_{2,2}^{(5)}.
\end{align*}
To estimate $I_{2,2}^{(3)}$  we note that  $\abs{z-x}\geq \frac{\abs{x}}{2}$ if $\abs{z}\leq\frac{\abs{x}}{2}$ and $\abs{x}\geq 1$, so that 
\begin{align*}
I_{2,2}^{(3)}
&\leq \ssgrklam{\frac52 \abs{x}}^{\beta'} e^{-\frac{\abs{x}^2}{4\phantom{n}}}\int_{\abs{z}\leq \frac{\abs{x}}{2}} \frac{1}{\abs{z}^{\beta'}}\,\dz\\
&= N\,
\abs{x}^{\beta'}\abs{x}^{2-\beta'}e^{-\frac{|x|^2}{4}} 
=N\,
\abs{x}^{2}e^{-\frac{\abs{x}^2}{4}}\le N<\infty,\nonumber
\end{align*}
where the last $N$ depends only on $\beta'$.
For the estimate of $I_{2,2}^{(4)}$, we can use the fact that $\abs{z-x}\leq \frac{\abs{x}}{2}$ implies $\abs{z}\geq \frac{\abs{x}}{2}$ and $\frac{\abs{x}+\abs{z}+1}{\abs{z}}\le \frac{7\abs{x}/2}{\abs{x}/2}=7$ since $\abs{x}\geq 1$. We obtain 
\begin{equation*}
I_{2,2}^{(4)}
\leq 
5^{\beta'}
\int_{\bR^2} e^{-\abs{z-x}^2}\,\dz
\leq N
<\infty,
\end{equation*}
with $N$ depending only on $\beta'$.
Lastly, since $\abs{x}\geq 1$, we have
\begin{align*}
I_{2,2}^{(5)}
&\leq 
\frac{1}{\nrklam{\abs{x}/2}^{\beta'}}
\int_{\abs{z}\geq \frac{\abs{x}}{2}} e^{-\abs{z-x}^2}(\abs{x}+\abs{z}+1)^{\beta'}\,\dz\\
&\leq N  \ssgrklam{\frac{2}{|x|}}^{\beta'}
\int_{\abs{z}\geq\frac{\abs{x}}{2}}e^{-\abs{z-x}^2}\sgrklam{\abs{x}^{\beta'}+\abs{(z-x)+x}^{\beta'}+1}\,\dz\\
&\leq  N   \ssgrklam{\frac{2}{|x|}}^{\beta'}
\int_{\abs{z}\geq\frac{\abs{x}}{2}}e^{-\abs{z-x}^2}\sgrklam{\abs{x}^{\beta'}+\abs{(z-x)}^{\beta'}+1}\,\dz\\
&\leq  N \int_{\bR^2}e^{-\abs{z-x}^2}\,\dz + 
N \int_{\bR^2}e^{-\abs{z-x}^2}\sgrklam{\abs{z-x}^{\beta'}+1}\,\dz\\
&\leq N <\infty, 
\end{align*}
where, again, the last constant only depends on $\beta'$.
By this, we finished proving that \eqref{homo 2015.06.09_3} holds for arbitrary $x\in\bR^2$ with a constant $N$ that depends only on $\beta'$, and, hence \eqref{homo 2015.06.09_1} holds with a constant depending only on $b$. Consequently, 
\[
\sup_{s\in(0,t)}\Abs{I_{2,2}(t,s,x)}
\leq N\,\abs{x}^{-2}
\]
with a constant $N$ that depends only on $\lambda$, $\mu$, and $\alpha$, $\beta$ from \eqref{condition 2015.06.10_2}, but not on $t$ or $x$. This is precisely what we wanted to prove for $p=2$.

\smallskip

\noindent\emph{The case $p>2$.} 
We have to show that 
\begin{equation}\label{estimate 2015.12.15_2}
\ssgrklam{\int^t_0 I_{2,p}^{\frac{2(p-1)}{p-2}}(t,s,x)\,\ds}^{\frac{p-2}{2}}
\leq N\, \abs{x}^{-2},
\end{equation}
with a constant $N\in(0,\infty)$ that depends only on $p$, $\lambda$, $\mu$, and $\alpha$, $\beta$ from \eqref{condition 2015.06.10_2}.
To this end, we first use   Minkowski's  inequality for integrals to obtain
\begin{align*}
&\ssgrklam{\int^t_0 I_{2,p}^{\frac{2(p-1)}{p-2}}(t,s,x)\,\ds}^{\frac{p-2}{2}}\\
&\quad\leq 
\reklam{\int_{\bR^2}\rrklam{\int^{t}_0 \frac{\abs{x}^{\frac{2p}{p-2}(\mu-2)}}{\abs{y}^{\frac{2p}{p-2}(\mu-1) }}\,R^{\frac{2p}{p-2}(\lambda-\alpha) }_{x,t-s}\,R^{\frac{2p}{p-2}(\lambda-\beta) }_{y,t-s}\frac{e^{-\frac{\abs{x-y}^2}{t-s\phantom{n}}}}{(t-s)^{\frac{2(p-1)}{p-2}}}\,\ds}^{\frac{p-2}{2(p-1)}}\dy}^{p-1}.
\end{align*}
For arbitrary $t$, $x$, and $y$, the inner integral
\begin{align*}
\mathcal{J}(t,x,y)
&:=
\int^{t}_0 \frac{\abs{x}^{\frac{2p}{p-2}(\mu-2)}}{\abs{y}^{\frac{2p}{p-2}(\mu-1) }}\,R^{\frac{2p}{p-2}(\lambda-\alpha) }_{x,t-s}\,R^{\frac{2p}{p-2}(\lambda-\beta) }_{y,t-s}\frac{e^{-\frac{\abs{x-y}^2}{t-s\phantom{n}}}}{(t-s)^{\frac{2(p-1)}{p-2}}}\,\ds\\
&\leq 
\int^{\infty}_0 
\frac{\abs{x}^{\frac{2p}{p-2}(\lambda-\alpha+\mu-2)}\abs{y}^{\frac{2p}{p-2}(\lambda-\beta-\mu+1) }}{(\abs{x}+\abs{y}+\sqrt{s})^{\frac{2p}{p-2}(2\lambda-\alpha-\beta)}}\frac{e^{-\frac{\abs{x-y}^2}{s}}}{s^{\frac{2(p-1)}{p-2}}}\,\ds,
\end{align*}
which, by the substitution $s=\abs{x-y}^2r$,  turns into
\begin{align*}
&\int^{\infty}_0 
\frac{\abs{x}^{\frac{2p}{p-2}(\lambda-\alpha+\mu-2)}\cdot \abs{y}^{\frac{2p}{p-2}(\lambda-\beta-\mu+1) }\cdot \abs{x-y}^{2-\frac{4(p-1)}{p-2}}}{(\abs{x}+\abs{y}+\abs{x-y}\sqrt{r})^{\frac{2p}{p-2}(2\lambda-\alpha-\beta)}}\frac{e^{-\frac{1}{r}}}{r^{\frac{2(p-1)}{p-2}}}\,\dr\\
&\qquad \qquad \qquad \leq
\abs{x}^{\frac{2p}{p-2}(\lambda-\alpha+\mu-2)}\cdot \abs{y}^{\frac{2p}{p-2}(\lambda-\beta-\mu+1) }\cdot \abs{x-y}^{2-\frac{4(p-1)}{p-2}}\\
&\qquad \qquad \qquad \qquad \qquad \qquad \qquad\times(\abs{x}+\abs{y})^{-\frac{2p}{p-2}(2\lambda-\alpha-\beta)}\cdot\Gamma(p/(p-2)),
\end{align*}
where $\Gamma(p/(p-2))<\infty$ is the Gamma function evaluated at $p/(p-2)$.  Hence,
\begin{align*}
\ssgrklam{\int^t_0 I_{2,p}^{\frac{2(p-1)}{p-2}}(t,s,x)\,\ds}^{\frac{p-2}{2(p-1)}}
&\leq 
\int_{\bR^2}\mathcal{J}^{\frac{p-2}{2(p-1)}}(t,x,y)\,\dy\\
&\leq
\abs{x}^{p'(\lambda-\alpha+\mu-2)} 
\int_{\bR^2}\frac{\abs{y}^{p'(\lambda-\beta-\mu+1)}}{\abs{x-y}^{p'}(\abs{x}+\abs{y})^{p'(2\lambda-\alpha-\beta)}}\,\dy\\
&=
\abs{x}^{p'(\lambda-\alpha+\mu-2) + p'(\lambda-\beta-\mu+1)-p'-p'(2\lambda-\alpha-\beta)+2}\\
&\qquad\qquad\qquad\times
\int_{\mathbb{R}^2}\frac{\abs{\eta}^{p'(\lambda-\beta-\mu+1)}}
{\Big\lvert\frac{x}{|x|}-\eta\Big\rvert^{p'}(1+\abs{\eta})^{p'(2\lambda-\alpha-\beta)}}\,\mathrm d\eta\nonumber\\
&=|x|^{-\frac{2}{p-1}}\int_{\mathbb{R}^2}\frac{|\eta|^{p'(\lambda-\beta-\mu+1)}}{|\xi-\eta|^{p'}(1+|\eta|)^{p'(2\lambda-\alpha-\beta)}}\,\mathrm d\eta,\nonumber
\end{align*}
where $\xi$ is any fixed point in $\R^2$ with $|\xi|=1$; the rotation invariance of the last integral can be seen by using the orthogonality of the rotation operator in $\R^2$.
Thus, in order to prove that \eqref{estimate 2015.12.15_2} holds with an appropriate constant, it is enough to check that, e.g., for  $\xi_0:=(1,0)$,
\begin{eqnarray}
\int_{\mathbb{R}^2}\frac{|\eta|^{p'(\lambda-\beta-\mu+1)}}{|\xi_0-\eta|^{p'}(1+|\eta|)^{p'(2\lambda-\alpha-\beta)}}\,\mathrm d\eta= N<\infty,\nonumber
\end{eqnarray}
with $N$ depending only on $p$, $\lambda$, $\mu$, $\alpha$, and $\beta$. \icp  To this end, we divide the integral into three parts:
\begin{align*}
\int_{\bR^2}\ldots \; 
=
\int_{\abs{\eta-\xi_0}<\frac12}\ldots\;+\int_{\frac12\leq \abs{\eta-\xi_0}<3}\ldots\;+\int_{3\leq \abs{\eta-\xi_0}}\ldots
=:
J_1+J_2+J_3.
\end{align*}
Since $p'=p/(p-1)<2$ as $p>2$, we have
\begin{equation*}
J_1
\leq N
\int_{\abs{\eta-\xi}<\frac12}\frac{1}{\abs{\eta-\xi}^{p'}}\,\mathrm d\eta
=N
\int^{1/2}_0 r^{-p'+1}\,\dr<\infty.
\end{equation*}
Also, $J_2$ is trivially bounded and
\begin{equation*}
J_3
\leq 
\int_{\abs{\eta-\xi}\geq 2}\abs{\eta}^{p'(\lambda-\beta-\mu+1)-p'-p'(2\lambda-\alpha-\beta)}\,\mathrm d\eta
\leq N 
\int^{\infty}_2 r^{-p'(\lambda-\alpha+\mu)+1}dr
<\infty
\end{equation*}
since $p'(\lambda-\alpha+\mu)>2$ by \eqref{condition 2015.06.10_2}. 
As these constants depend only on $p$, $\lambda$, $\mu$, $\alpha$, and $\beta$, we have proven that Estimate~\eqref{estimate 2015.12.15_2} holds with an appropriate constant.

\smallskip

\noindent{\bf Step 3.} Estimate~\eqref{estimate 2015.12.15_1} proven in the first step, together with the estimate from the second step imply 
\begin{align*}
\E
\Abs{\mathcal{G}_2h(t,x)}^2
&\leq N\, 
\E \int^t_0 I_{1,p}(t,s,x)\;\abs{x}^{-2}\,\ds\\
&= N\, 
\E \int^t_0 \int_\cD e^{-\frac{\abs{x-y}^2}{t-s\phantom{n}}}\abs{h(s,y)}^p_{\ell_2}\,R^{\alpha p}_{x,t-s}\,R^{\beta p}_{y,t-s}\frac{1}{t-s}\,\dy \,\abs{x}^{-2}\,\ds,
\end{align*}
for each $t\in(0,T]$  and $x\in {\mathcal{D}}$, where the constant $N\in(0,\infty)$ depends only on $p$, $\lambda$, $\mu$, $\alpha$ and $\beta$.
Therefore, by enlarging the domain of integration and using Fubini's theorem, we obtain
\begin{align*}
\E &\nnrm{\mathcal{G}_2 h}{L_p((0,T]\times\cD)}^p\\
&\quad=
\E \int_0^T \int_\cD \Abs{\mathcal{G}_2h(t,x)}^2 \,\dx\,\dt\\
&\quad\leq N\,
\E \int_0^T \int_\cD \int^t_0 \int_\cD e^{-\frac{\abs{x-y}^2}{t-s\phantom{n}}}\abs{h(s,y)}^p_{\ell_2}\,R^{\alpha p}_{x,t-s}\,R^{\beta p}_{y,t-s}\frac{1}{t-s}\,\dy \,\abs{x}^{-2}\,\ds\,\dx\,\dt\\
&\quad\leq N\,
\E \int_0^T \int_\cD  \ssgrklam{\int_{\bR^2} \int_s^T e^{-\frac{\abs{x-y}^2}{t-s\phantom{n}}}\,R^{\alpha p}_{x,t-s}\,R^{\beta p}_{y,t-s}\frac{\abs{x}^{-2}}{t-s} \,\,\dt \,\dx} \abs{h(s,y)}^p_{\ell_2}\,\dy\,\ds
\end{align*}
with the same constant as above. In order to finish the proof, it is enough to show that
\begin{equation}\label{eq:finalstep:1}
\sup_{\substack{s\in(0,T]\\ y\in\cD}}\ssgrklam{\int_{\bR^2} \int_s^T e^{-\frac{\abs{x-y}^2}{t-s\phantom{n}}}\,R^{\alpha p}_{x,t-s}\,R^{\beta p}_{y,t-s}\frac{\abs{x}^{-2}}{t-s} \,\dt \,\dx}
\leq N <\infty,
\end{equation}
for a constant $N$ that depends only on $p$, $\alpha=\alpha(p,\lambda,\mu)$, and $\beta=\beta(p,\lambda,\mu)$. 
This can be checked as follows: Using Fubini's theorem, the change of variable $x=y-z\sqrt{t-s}$ and the fact that $\beta > 0$ due to~\eqref{condition 2015.06.10_2}, we obtain
\begin{align*}
\int_{\bR^2} & \int_s^T  e^{-\frac{\abs{x-y}^2}{t-s\phantom{n}}}\,R^{\alpha p}_{x,t-s}\,R^{\beta p}_{y,t-s}\frac{\abs{x}^{-2}}{t-s} \,\dt \,\dx \\
&=
\int_s^T \int_{\bR^2} e^{-\frac{\abs{x-y}^2}{t-s\phantom{n}}}\,R^{\alpha p}_{x,t-s}\,R^{\beta p}_{y,t-s}\frac{\abs{x}^{-2}}{t-s} \,\dx \,\dt \\
&=
\int_s^T \int_{\bR^2} e^{-\frac{\abs{x-y}^2}{t-s\phantom{n}}} 
\frac{\abs{x}^{\alpha p} \cdot \abs{y}^{\beta p} }{(\abs{x}+\abs{y}+\sqrt{t-s})^{(\alpha+\beta) p}}
\frac{\abs{x}^{-2}}{t-s} \,\dx \,\dt \\
&=
\int_s^T \int_{\bR^2} e^{-\abs{z}^2} \frac{\abs{y-z\sqrt{t-s}}^{\alpha p-2} \cdot \abs{y}^{\beta p} }{(\abs{y-z\sqrt{t-s}}+\abs{y}+\sqrt{t-s})^{(\alpha+\beta) p}} \,\dz \,\dt \\
&\leq
\int_s^T \int_{\bR^2} e^{-\abs{z}^2} \frac{\abs{y-z\sqrt{t-s}}^{\alpha p-2}}{(\abs{y-z\sqrt{t-s}}+\abs{y}+\sqrt{t-s})^{\alpha p-2}} \,\dz 
\frac{\abs{y}^{\beta p}}{(\abs{y}+\sqrt{t-s})^{\beta p +2}}
\,\dt.
\end{align*}
As $\alpha p -2 \in (-2,\infty)$ due to~\eqref{condition 2015.06.10_2}. we can now use the fact that~\eqref{homo 2015.06.09_1} holds with a constant that depends only on $b\in(-2,\infty)$ to estimate the inner integral  by a finite constant depending only on $\alpha$ and $p$. 
The remainder can be estimated by a finite constant depending only on $\beta$ and $p$, since $\beta p +2 > 2$ due to~\eqref{condition 2015.06.10_2}, and therefore
\begin{align*}
\int_s^T
\frac{\abs{y}^{\beta p}}{(\abs{y}+\sqrt{t-s})^{\beta p +2}}
\,\dt
&\leq
\int_0^\infty \frac{\abs{y}^{\beta p}}{(\abs{y}+\sqrt{t})^{\beta p +2}}
\,\dt = 
\int_0^\infty \frac{1}{(1+\sqrt{\tau})^{\beta p+2}}\,\mathrm{d}\tau
<\infty,
\end{align*}
where we used the change of variable $t=\abs{y}^2\tau$.
Thus, we have found a constant depending only on $p$, $\alpha$, and $\beta$, so that estimate~\eqref{eq:finalstep:1} holds. 
Consequently, the theorem is proved.
\end{proof}

We conclude this section with a technical remark concerning the measurability of the stochastic convolution \eqref{eq:def:conv:stoch}.

\begin{remark}\label{rem:existence_stoch_conv}
In the course of the proof of Theorem~\ref{thm stochastic part} we have implicitly shown that, for almost all $(t,x)\in(0,T]\times\cD$, the (series of) stochastic integral(s)
\begin{equation}\label{eq:rem_v(t,x)}
w(t,x)=\sum_{k=1}^\infty\int_0^t\int_\cD G(t-s,x,y)g^k(s,y)\,\dy\,\dw^k_s
\end{equation} 
is well-defined as an element of $L_2(\Omega)$. This follows from the estimate
\[
\int_0^T \int_\cD 
\rrklam{
\E
\sum_{k=1}^\infty\int_0^t
\ssgrklam{\int_\cD \abs{x}^{\mu-2}G(t-s,x,y)g^k(s,y)\,\dy}^2\ds}\dx\,\dt<\infty\]
and Itô's isometry.
Using arguments similar to the ones used in the standard proofs of the stochastic Fubini theorem, one can verify that there exists a predictable (i.e., $\cP_T\otimes\cB(\cD)$-measurable) function $\overline{w}\colon\Omega\times(0,T]\times\cD\to\bR$ such that 
\[
\overline w(t,x)\stackrel{\text{a.s.}}=w(t,x)\;\;\text{ for almost all }(t,x)\in (0,T]\times\cD.\]
A modification of the arguments in the proof of Theorem~\ref{thm stochastic part} even shows that, for all $t\in(0,T]$,
\[\int_\cD
\rrklam{\E\sum_{k=1}^\infty\int_0^t\ssgrklam{\int_\cD \abs{x}^{\mu-1}G(t-s,x,y)g^k(s,y)\,\dy}^2\ds}\dx<\infty.\]
Consequently, for \emph{all} $t\in(0,T]$, $w(t,x)\in L_2(\Omega)$ is well-defined by \eqref{eq:rem_v(t,x)} for almost all $x\in\cD$. Using the function $\overline{w}$ from above, we can therefore construct a function $\widehat w\colon\Omega\times(0,T]\times\cD\to\bR$ such that
\begin{itemize}
\item for \emph{all} $t\in(0,T]$: $\widehat w(t,x)\stackrel{\text{a.s.}}=w(t,x)\;\;\text{for almost all }x\in\cD;$
\item the mapping 
\[
\Omega\times(0,T]\ni(\omega,t)\mapsto\widehat w(\omega,t,\cdot)\in L_{2,\text{loc}}(\cD)
\] is $\cP_T^{\prob\otimes\dt}/\cB(L_{2,\text{loc}}(\cD))$-measurable.
\end{itemize}
Here, $L_{2,\text{loc}}(\cD)$ denotes the space of locally square-integrable functions on $\cD$.
In order to obtain such a $\widehat w$, we define $\widehat w(t,\cdot):=\overline{w}(t,\cdot)$ for all $t$ such that $\overline{w}(t,x)\stackrel{\text{a.s.}}=w(t,x)$ for almost all $x$; for each remaining $t$ we use an approximation argument to construct an $\cF_t\otimes\cB(\cD)$-measurable function $\widehat w(t,\cdot)\colon\Omega\times\cD\to\bR$ satisfying $\widehat w(t,x)\stackrel{\text{a.s.}}=w(t,x)$ for almost all $x$. We remark that the $L_{2,\text{loc}}(\cD)$-valued function $(\omega,t)\mapsto\widehat w(\omega,t,\cdot)$ constructed that way is pseudo-predictable in the sense of \cite[Section~III.5]{Kry1995}.

Whenever we deal with a stochastic convolution of the form \eqref{eq:rem_v(t,x)}, we consider either a version $\overline{w}$ or a version $\widehat w$ as above, depending on what is more suitable in the current context. We will usually not explicitly indicate the choice of the version.  
\end{remark}


\mysection{Weighted Sobolev regularity of the stochastic heat equation}\label{sec:WSobReg}

\noindent In this section we use our main estimate~\eqref{result stochastic part} to establish existence and uniqueness of a solution to the stochastic heat equation
\begin{equation}\label{eq:heatEq:st}
\left.
\begin{alignedat}{3}
\du 
&= 
\grklam{ \Delta && u + f}\,\dt
+
g^k\,\dw^k_t \quad \text{on } \Omega\times(0,T]\times\cD,	\\
u
&=
0 && \quad \text{on } \Omega\times(0,T]\times\partial\cD,	\\
u(0)
&=
0 && \quad \text{on } \Omega\times\cD,
\end{alignedat}
\right\}	
\end{equation}
within suitable weighted $L_p$-Sobolev spaces of order one. Thereby the free terms $f$ and $g$ are allowed to have certain singularities at the vertex of $\cD$. 
Our analysis takes place within the framework of the analytic approach to SPDEs initiated in \cite{Kry1999} by N.V.~Krylov.
In particular, we borrow the `distributional' solution concept from there and show that the sum $u=w+v$ of the stochastic convolution $w$ introduced in \eqref{eq:def:conv:stoch} and the deterministic convolution $v$ from \eqref{eq:def:conv:det} is a solution to Eq.~\eqref{eq:heatEq:st} in this sense. 

The solution and the forcing terms in Equation~\eqref{eq:heatEq:st} depend on $\omega\in\Omega$, $t\in (0,T]$, and $x\in\cD$. In what follows, we take a functional analytic point of view and look at these terms as functions on $\Omega_T=\Omega\times (0,T]$ taking values in suitable subspaces of the space of Schwartz distributions on $\cD$. 
Estimate~\eqref{result stochastic part} indicates that weighted Sobolev spaces, where the weights are appropriate powers of the distance $\rho_o(x):=\lvert x\rvert$ of a point $x\in\cD$ to the origin, are a reasonable choice for these subspaces. 
Therefore, we introduce the following notation.
Fix $\theta\in\bR$ and $p\in(1,\infty)$. 
We write
\[
L^{[o]}_{p,\theta}(\cD):=L_{p}(\cD,\cB(\cD),\rho_o^{\theta-2}\mathrm{d}x;\R)
\quad \text{and} \quad 
L^{[o]}_{p,\theta}(\cD;\ell_2):=L_p(\cD,\cB(\cD),\rho_o^{\theta-2}\mathrm{d}x;\ell_2)
\]
for the weighted $L_p$-spaces of real-valued and $\ell_2$-valued functions with weight $\rho_o^{\theta-2}$ and denote by 
\[
\wso^1_{p,\theta}(\cD)
:=
\ssggklam{
f : \nnrm{f}{\wso^1_{p,\theta}(\cD)} := \sum_{\abs{\alpha}\leq 1} 
\gnnrm{\rho_o^{\abs{\alpha}} D^\alpha f}{L_{p,\theta}^{[o]}(\cD)}
< 
\infty
}
\]
the corresponding weighted $L_p$-Sobolev space of order one. Note that the weight depends on the order of differentiability.
Furthermore, we write $\mathring{\wso}^1_{p,\theta}(\cD)$ for the closure of the test functions $\cont^\infty_0(\cD)$ in $\wso^1_{p,\theta}(\cD)$, i.e., 
\[
\grklam{\mathring{\wso}^1_{p,\theta}(\cD),\nnrm{\cdot}{\mathring{\wso}^1_{p,\theta}(\cD)}}
:=
\sgrklam{\overline{\cont^\infty_0(\cD)}^{\nnrm{\cdot}{\wso^1_{p,\theta}(\cD)}}, \nnrm{\cdot}{\wso^1_{p,\theta}(\cD)}}.
\]
The dual of $\mathring{\wso}^1_{p,\theta}(\cD)$ is denoted by
\begin{equation}\label{eq:dual:wso}
\wso^{-1}_{p',\theta'}(\cD)
:=
\grklam{\mathring{\wso}^1_{p,\theta}(\cD)}^*,
\quad
\frac{1}{p}+\frac{1}{p'}=1,
\quad
\frac{\theta}{p}+\frac{\theta'}{p'}=2.
\end{equation}
This space consists of (unique) extensions of generalized functions on $\cD$, i.e., if $u\in \wso^{-1}_{p',\theta'}(\cD)$ and $\varphi\in \cont_0^\infty(\cD)\subset \wso^1_{p,\theta}(\cD)$, then $u(\varphi)=(u,\varphi)$.
Clearly, the dual of $L^{[o]}_{p,\theta}(\cD)$ is isomorphic to $L^{[o]}_{p',\theta'}(\cD)$, if $p'$ and $\theta'$ are chosen as above.
To shorten some statements, we will occasionally write 
\[
K^0_{p,\theta}(\cD):=L^{[o]}_{p,\theta}(\cD)
\quad\text{and}\quad
K^0_{p,\theta}(\cD;\ell_2):=L^{[o]}_{p,\theta}(\cD;\ell_2),
\]
and call these spaces weighted $L_p$-Sobolev spaces of order zero.

Before we proceed, we record some basic properties of these spaces. 
The proofs are left to the reader.

\begin{lemma}\label{lem:wS:property}
Let $p\in(1,\infty)$ and $\theta\in\bR$.

\smallskip

\noindent\textbf{\textup{(i)}}  The spaces $L^{[o]}_{p,\theta}(\cD)$, $\wso^{1}_{p,\theta}(\cD)$,  $\mathring{\wso}^1_{p,\theta}(\cD)$ and $\wso^{-1}_{p,\theta}(\cD)$ are Banach spaces. 

\smallskip

\noindent\textbf{\textup{(ii)}}  For $i\in\{1,2\}$ and $n\in\{0,1\}$, the operator $\rho_o D_{x^{i}}\colon \wso^{n}_{p,\theta}(\cD)\to \wso^{n-1}_{p,\theta}(\cD)$, $u\mapsto\rho_o u_{x^{i }}$, is well-defined, linear and bounded.
\end{lemma}

\begin{remark}\label{rem:zero:bdry}
By using localisation techniques together with the corresponding results for classical (unweighted) Sobolev spaces, one can prove that $\mathring{\wso}^1_{p,\theta}(\cD)$ consists of all functions in $\wso^1_{p,\theta}(\cD)$ that vanish on the boundary of the underlying domain $\cD$ in the sense that their boundary trace---which is defined almost everywhere on $\partial\cD\setminus\{0\}$ with respect to the corresponding Hausdorff measure---equals zero.
In particular, if we fix a smooth function $\zeta\in\cont^\infty(\overline{\cD})$ with support within the closure of the stripe
\[
\cD^{(m,M)}
:=
\ggklam{x\in\cD : m < \abs{x} < M}
\]
for some $0<m<M<\infty$, then  $u\in \mathring{K}^1_{p,\theta}(\cD)$ implies
$\zeta u \in \mathring{W}^1_p(\cD^{(m,M)})=\overline{\cont_0^\infty(\cD^{(m,M)})}^{\nnrm{\cdot}{W^1_p(\cD^{(m,M)})}}$, and
\[
\nnrm{\zeta u}{W^1_p(\cD^{(m,M)})}
\leq N\,
\nnrm{u}{\wso^1_{p,\theta}(\cD)},
\]
with a constant $N$ that depends only on $\zeta$, $m$, $M$, $\theta$, and $p$.
\end{remark}

The weighted Sobolev spaces introduced above are used to measure the regularity  of the solution and the free terms in Eq.~\eqref{eq:heatEq:st} in the space variable $x\in\cD$. With respect to $(\omega,t)\in\Omega_T$ we only require predictability and $p$-integrability. For notational clarity we introduce the following spaces of $p$-integrable predictable stochastic processes with values in the weighted Sobolev spaces introduced above. 
Let $T>0$. Recall that we write $\pred_T$  for the predictable $\sigma$-algebra on $\Omega_T$ generated by the given filtration
$(\cF_{t})_{t\in[0,T]}$.
For $\theta\in\R$, $1<p<\infty$ and $n\in\{-1,0,1\}$, we use the abbreviations
$$
\bwso^{n}_{p,\theta}(\cD,T)
:=
L_p(\Omega_T, \pred_T, \prob\otimes\mathrm dt;\wso^{n}_{p,\theta}(\cD)), 
\quad 
\bL^{[o]}_{p,\theta}(\cD,T):=\bwso^0_{p,\theta}(\cD,T),
$$
and
\[
\mathring{\bwso}^1_{p,\theta}(\cD,T)
:=
L_p(\Omega_T, \pred_T, \prob\otimes\mathrm dt;\mathring{\wso}^{1}_{p,\theta}(\cD)),
\]
as well as 
$$
\bL^{[o]}_{p,\theta}(\cD,T;\ell_2)
:=
L_p(\Omega_T, \pred_T, \prob\otimes\dt;L^{[o]}_{p,\theta}(\cD; \ell_2)),
$$
for the spaces of predictable, $p$-Bochner integrable stochastic processes taking values in the weighted $L_p$-Sobolev spaces introduced above. 
Finally, we write $\bwso^{\infty}_{0}(\cD,T;\ell_2)$ for the space of all functions $g=(g^k)_{k\in\bN}\colon\Omega_T\times\cD\to\ell_2$ 
of the form
\begin{equation}\label{eq:noise:simple}
g=\sum_{k=1}^K g^k\,\mathrm e_k, \qquad
g^k(\omega,t,x)= \sum_{i=1}^{n(k)}
{\bf 1}_{(\tau^{(k)}_{i-1}(\omega),\tau^{(k)}_i(\omega)]}(t) \, g^{ik}(x),
\end{equation}
where $(\mathrm{e}_k)_{k\in\bN}$ is the standard orthonormal basis of $\ell_2$, $K\in\bN$, and for each $k=1,\ldots,K$, $(\tau^{(k)}_{i})_{i=0,\ldots, n(k)}$ is a finite non-decreasing sequence of bounded stopping times with respect to the filtration $(\cF_t)_{t\geq0}$, and $g^{ik}\in \cont^\infty_0(\cD)$, $i=1,\ldots, n(k)$, for some $n(k)\in\bN$. 

\begin{remark}
Since for arbitrary $p\in (1,\infty)$ and $\theta\in\bR$ the smooth and compactly supported functions $\cont^\infty_0(\cD)$  are dense in $L^{[o]}_{p,\theta}(\cD)$ and, by definition, in $\mathring{\wso}^1_{p,\theta}(\cD)$, one can deduce that $\bwso^\infty_0(\cD,T;\ell_2)$ is a dense subspace of $\bL^{[o]}_{p,\theta}(\cD,T;\ell_2)$ and of $\mathring{\bwso}^1_{p,\theta}(\cD,T)$. 
\end{remark}

With the help of these spaces we can introduce corresponding classes of stochastic processes that are tailor-made for the development of an $L_p$-theory for second-order SPDEs with zero Dirichlet boundary condition. 

\begin{defn}\label{def:frH}
Let $p\geq2$ and $\theta\in\bR$.
We write $u \in\frwso^{1}_{p,\theta}(\cD,T)$  if
$u\in\mathring{\bwso}^1_{p,\theta-p}(\cD,T)$
 and
there exist 
$f\in \bwso^{-1}_{p,\theta+p}(\cD,T)$ and
 $g\in \bL^{[o]}_{p,\theta}(\cD,T;\ell_2)$
such that 
\begin{equation}\label{eqn 28}
\mathrm du=f\,\mathrm dt +g^k \,\mathrm  dw^k_t
\end{equation}
\emph{in the sense of distributions with} $u(0,\cdot)=0$. 
That is, for any $\varphi \in
\cont^{\infty}_{0}(\cD)$, with probability one, the equality
\begin{equation}\label{eq:distribution}
(u(t,\cdot),\varphi)=   \int^{t}_{0}
(f(s,\cdot),\varphi) \, \mathrm ds + \sum^{\infty}_{k=1} \int^{t}_{0}
(g^k(s,\cdot),\varphi)\, \mathrm dw^k_s
\end{equation}
holds for all $t \leq T$. In this situation
we also write
$$
\bD u:=f\qquad\text{and}\qquad \bS u :=g
$$
for the \emph{deterministic} and the \emph{stochastic part}, respectively.
The norm in  $\frwso^{1}_{p,\theta}(\cD,T)$ is
defined as 
\begin{equation}\label{eq:frH:nnrm}
\nnrm{u}{\frwso^{1}_{p,\theta}(\cD,T)}:=
\nnrm{u}{\bwso^{1}_{p,\theta-p}(\cD,T)} +
\nnrm{\bD u}{\bwso^{-1}_{p,\theta+p}(\cD,T)} +
\nnrm{\bS u}{\bL^{[o]}_{p,\theta}(\cD,T;\ell_2)}.
\end{equation}
\end{defn}

\begin{remark}
\emph{(i)}
Since the stochastic part $\bS u=g$ in Definition~\ref{def:frH} belongs to $\bL^{[o]}_{p,\theta}(\cD,T;\ell_2)$ for some $p\geq 2$, the series $\sum^{\infty}_{k=1} \int^{\cdot}_{0}
(g^k(s,\cdot),\varphi)\, \mathrm dw^k_s$ on the right hand side of~\eqref{eq:distribution} converges in $L_2(\Omega;\cont([0,T];\bR))$.
This can be proven by following the lines of \cite[Remark~3.2]{Kry1999} and using the fact that $g\in \bL^{[o]}_{p,\theta}(\cD,T;\ell_2)$ and $\cont_0^\infty(\cD)\subset L^{[o]}_{q,\tau}(\cD)$ for arbitrary  $q>1$ and $\tau\in\bR$. 

\smallskip

\noindent\emph{(ii)} The deterministic and the stochastic part in Definition~\ref{def:frH} are uniquely determined by $u\in\frwso^1_{p,\theta}(\cD,T)$.
This can be seen by using the same arguments as in \cite[Remark~3.3]{Kry1999}.

\smallskip

\noindent\emph{(iii)} $\frwso^{1}_{p,\theta}(\cD,T)$ is a Banach space for any $p\geq2$ and $\theta\in\bR$.
Indeed, the norm~\eqref{eq:frH:nnrm} is well-defined on $\frwso^{1}_{p,\theta}(\cD,T)$ due to \emph{(ii)},
and the completeness follows, for instance, by 
using the corresponding result on the whole space $\bR^2$ from \cite[Theorem~3.7]{Kry1999} and following the localization strategy suggested in \cite[Remark~3.8]{Kry2001}.
\end{remark}

Now we are able to define rigorously what we mean by a solution to the stochastic heat equation \eqref{eq:heatEq:st}. 

\begin{defn}\label{def:solution:dist}
Let $p\geq 2$ and $\theta\in\bR$.
A stochastic process $u\in\mathring{\bwso}^{1}_{p,\theta-p}(\cD)$
is a \emph{solution of Eq.~\eqref{eq:heatEq:st} in the class $\frwso^{1}_{p,\theta}(\cD,T)$} 
if
$u\in \frwso^{1}_{p,\theta}(\cD,T)$ with
\[
\bD u = \Delta u + f
\qquad
\text{and}
\qquad
\bS u = (g^k)_{k\in\bN}
\]
in the sense of Definition~\ref{def:frH}.
\end{defn}

All these notions and facts at hand, we can state our main result concerning the solvability of the stochastic heat equation~\eqref{eq:heatEq:st}.

\begin{thm}\label{thm:main:eq}
Let $p\geq 2$ and let $\theta\in\bR$ fulfil
\begin{equation}\label{eq:thetarange:det}
p\ssgrklam{1-\frac{\pi}{\kappa_0}}
<
\theta
<
p\ssgrklam{1+\frac{\pi}{\kappa_0}}.
\end{equation}
Assume that 
$f\in \bL^{[o]}_{p,\theta+p}(\cD,T)$
and 
$g\in\bL^{[o]}_{p,\theta}(\cD,T;\ell_2)$.
Then
\[
u(t,x)
:=
\int_0^t \int_\cD G(t-s,x,y) f(s,y)\,\dy\,\ds
+
\sum_{k=1}^\infty
\int_0^t \int_\cD G(t-s,x,y) g^k(s,y)\,\dy\,\dw^k_s
\]
is the unique solution in the class $\frwso^{1}_{p,\theta}(\cD,T)$ to Eq.~\eqref{eq:heatEq:st}.
Moreover,
\begin{equation}\label{eq:heatEq:estimate}
\nnrm{u}{\frwso^{1}_{p,\theta}(\cD,T)}
\leq N
\sgrklam{
\nnrm{f}{\bL^{[o]}_{p,\theta+p}(\cD,T)}
+
\nnrm{g}{\bL^{[o]}_{p,\theta}(\cD,T;\ell_2)}
}
\end{equation}
with a constant $N\in(0,\infty)$ that does not depend on $u$, $f$, $g$, or $T$.
\end{thm}

\begin{remark}\label{rem:sol+unique:det}
Let $f\in L_p((0,T];L^{[o]}_{p,\theta+p}(\cD))$ with $p\geq 2$ and $\theta$ fulfilling the condition~\eqref{eq:thetarange:det}.
Then, the estimate \eqref{result deterministic part} from \cite{Sol2001} for the deterministic convolution $v$ in \eqref{eq:def:conv:det} can be rewritten as 
\begin{equation}\label{eq:est:Sollonikov}
\int_{0}^T \nnrm{v(t,\cdot)}{\wso^2_{p,\theta-p}(\cD)}^p\,\mathrm dt
\leq
N
\int_0^T \nnrm{f(t,\cdot)}{L^{[o]}_{p,\theta+p}(\cD)}^p\,\mathrm{d}t,
\end{equation}
where, for $\tilde\theta\in\bR$,
\[
\wso^2_{p,\tilde\theta}(\cD)
:=
\ssggklam{
f : \nnrm{f}{\wso^2_{p,\tilde\theta}(\cD)} := \sum_{\abs{\alpha}\leq 2} 
\gnnrm{\rho_o^{\abs{\alpha}} D^\alpha f}{L_{p,\tilde\theta}^{[o]}(\cD)}
< 
\infty
}.
\]
In particular, $v$ is an element of $L_p((0,T];\wso^2_{p,\theta-p}(\cD))$ and $N$ can be chosen to be independent of $T$, see \cite[Theorem~1.2]{Sol2001}.
Moreover, it can be proven that $v$ 
is the unique solution in $L_p((0,T];\mathring{\wso}^1_{p,\theta-p}(\cD))$ of the deterministic heat equation 
\begin{equation}\label{eq:heatEq:det}
\left.
\begin{alignedat}{3}
\frac{\partial}{\partial t} v(t,x) 
&= 
\Delta && v(t,x) + f(t,x),
\quad (t,x)\in (0,T]\times\cD,	\\
v(0,x)
&=
0, && \quad x\in \cD,\\
v(t,x)&=0, && \quad (t,x)\in(0,T]\times\partial D,
\end{alignedat}
\right\},
\end{equation} 
in the following sense: $v$ is the unique element in $L_p((0,T];\mathring{\wso}^1_{p,\theta-p}(\cD))$ 
that has a version $\bar{v}$ 
fulfilling
\[
\grklam{\bar{v}(t,\cdot),\varphi}
=
\int_0^t \grklam{\Delta \bar{v}(s,\cdot)+f(s,\cdot),\varphi} 
\,\ds
\]
for all $t\in(0,T]$ and all $\varphi\in\cont^\infty_0(\cD)$. 
\end{remark}

\begin{remark}\label{rem:uniqueness:det-stoch}
Note that the uniqueness of a solution in $L_p((0,T];\mathring{\wso}^1_{p,\theta-p}(\cD))$ to Eq.~\eqref{eq:heatEq:det} immediately implies the uniqueness of a solution in the class $\frwso^1_{p, \theta}(\cD,T)$ of the stochastic heat equation~\eqref{eq:heatEq:st}---whenever it exists. 
\end{remark}

In the remainder of this section we prove Theorem~\ref{thm:main:eq}. We start with the following existence and uniqueness result for Eq.~\eqref{eq:heatEq:st} with a noise term that is nice enough, i.e., with $g$ having a simple product structure and being smooth enough. In this case, we can prove the following assertion by applying the results mentioned in Remark~\ref{rem:sol+unique:det} to an appropriate deterministic heat equation.

\begin{lemma}\label{lem:existence:smooth}
Let $p\geq2$ and $\theta\in\bR$ fulfil the condition~\eqref{eq:thetarange:det}. 
Furthermore, assume that $g\in \bwso^{\infty}_{0}(\cD,T;\ell_2)$.
Then the stochastic convolution 
\begin{equation}\label{eqn 12.03.4}
w(t,x)
=
\sum_{k=1}^\infty\int^t_0 \int_{\cD} G(t-s,x,y)g^k(s,y)\,\mathrm dy \,\mathrm dw^k_s   
\end{equation}
is the unique solution in the class $\frwso^1_{p,\theta}(\cD,T)$ of Eq.~\eqref{eq:heatEq:st} with $f\equiv 0$.
\end{lemma}

\begin{proof}
Let $g\in\bwso^\infty_0(\cD,T;\ell_2)$ be of the form~\eqref{eq:noise:simple}. 
First note that, due to the smoothness of $g$ and its simple product structure, the function
\[
(\omega,t,x)
\mapsto
\widetilde{w}(\omega,t,x)
:=
\sum_{k=1}^K\sum_{i=1}^{n(k)} (w^k_{\tau^k_i(\omega)\land t}(\omega)-w^k_{\tau_{i-1}^k(\omega)\land t}(\omega))\,g^{ik}(x),
\]
is a $\pred_T\otimes\cB(\cD)$-measurable version  of the series 
$$
\sum_{k=1}^\infty\int^\cdot_0 g^k(s,\cdot)\,\mathrm dw^k_s 
$$ 
of stochastic integrals. For each $(\omega,t)\in\Omega_T$ and all multi-indexes $\alpha=(\alpha^1,\ldots,\alpha^d)\in\bN_0^d$, $D^\alpha_x\widetilde{w}(\omega,t,\cdot)$ is well-defined. Moreover, it has compact support in $\cD$ and for arbitrary $\tilde{p}\geq 2$ and $\tilde{\theta}\in\bR$,
$$
D^{\alpha}_x \widetilde{w}(\omega,\cdot,\cdot)\in L_{\tilde{p}}((0,T];L^{[o]}_{\tilde{p},\tilde{\theta}}(\cD)), \quad \omega\in\Omega,
$$
and $\widetilde{w}\in \mathring{\bwso}^1_{\tilde{p},\tilde{\theta}}(\cD,T)$.
In particular, for every fixed $\omega\in\Omega$, 
\begin{align*}
\overline{w}(&\omega,t,x)
:=
\int_0^t\int_\cD G(t-s,x,y) \Delta_y\widetilde{w}(\omega,s,y)\,\dy\,\ds	
\end{align*}
is the unique solution in $L_p((0,T];\mathring{\wso}^1_{p,\theta-p}(\cD)\cap \wso^2_{p,\theta-p}(\cD))$ to Eq.~\eqref{eq:heatEq:det} with forcing term
\[
f
=
\Delta_x\widetilde{w}(\omega,t,x)
=
\sum_{k=1}^K\sum_{i=1}^{n(k)}
(w_{\tau_i^k(\omega)\land t}^k(\omega)-w_{\tau^k_{i-1}(\omega)\land t}^k(\omega))
\Delta_x g^{ik}(x),
\]
cf.~Remark~\ref{rem:sol+unique:det}. 
Using the simple structure and smoothness of $\Delta_x\widetilde{w}$ together with the continuity of the  Green function with respect to the time variable, it can be checked that $\overline{w}$ is $\pred_T\otimes\cB(\cD)$-measurable, and that for every $\omega\in\Omega$, for arbitrary $\varphi\in\cont^\infty_0(\cD)$ and $t\in(0,T]$,
\begin{align*}
\grklam{\overline{w}(\omega,t,\cdot),\varphi}
=
\int_0^t
\grklam{\Delta\overline{w}(\omega,s,\cdot) + \Delta\widetilde{w}(\omega,s,\cdot),\varphi}\,\ds.
\end{align*}
Thus, 
\begin{align*}
\grklam{\overline{w}(\omega,t,\cdot)+\widetilde{w}(\omega,t,\cdot),\varphi}
=
\int_0^t
\grklam{\Delta\overline{w}(\omega,s,\cdot) + \Delta\widetilde{w}(\omega,s,\cdot),\varphi}\,\ds
+
\grklam{\widetilde{w}(\omega,t,\cdot),\varphi}.
\end{align*}
Consequently, $\overline{w}+\widetilde{w}$ is the unique solution in the class $\frwso^1_{p,\theta}(\cD,T)$ to Eq.~\eqref{eq:heatEq:det} with vanishing deterministic forcing term $f$, see also Remark~\ref{rem:uniqueness:det-stoch}.

In order to complete our proof, we only have to check that $\overline{w}+\widetilde{w}$ is a version of the stochastic convolution $w$. But this can be seen from the basic properties of the Green function and the stochastic Fubini theorem, which can be used to show that for almost all $(t,x)\in (0,T]\times\cD$, 
\begin{align*}
\sum_{k=1}^\infty&\int^t_0 \int_{\cD} G(t-s,x,y)g^k(s,y)\,\mathrm dy \,\mathrm dw^k_s
-
\sum_{k=1}^\infty\int^t_0 g^k(s,x)\,\mathrm dw^k_s\\
&=
\sum_{k=1}^K
\int_0^t
\sum_{i=1}^{n(k)}
\one_{(\tau_{i-1}^{(k)},\tau_i^{(k)}]}(r)
\ssgrklam{
\int_\cD G(t-r,x,y)g^{ik}(y)\,\mathrm{d}y - g^{ik}(x)
}
\,\mathrm{d}w^k_r\\
&=
\sum_{k=1}^K
\int_0^t
\sum_{i=1}^{n(k)}
\one_{(\tau_{i-1}^{(k)},\tau_i^{(k)}]}(r)
\ssgrklam{\int_0^{t-r}\int_\cD G(s,x,y)\Delta_y g^{ik}(y)\,\mathrm{d}y\,\mathrm{d}s} 
\,\mathrm{d}w^k_r\\
&=
\int_0^t
\int_\cD
G(t-s,x,y)
\ssgrklam{
\sum_{k=1}^K
\sum_{i=1}^{n(k)}
\Delta_y g^{ik}(y)
\int_0^s
\one_{(\tau_{i-1}^{(k)},\tau_i^{(k)}]}(r)
\,\dw^k_r}
\,\dy
\,\ds
\\
&=
\int_0^t
\int_\cD
G(t-s,x,y)
\Delta_y
\ssgrklam{
\sum_{k=1}^\infty
\int_0^s
g^k(r,y)
\,\dw^k_r}
\,\dy
\,\ds 
\qquad
\prob\textrm{-a.s.}	\qedhere
\end{align*}
\end{proof}

Next, we prove an a-priori estimate that holds if $g\in\bwso^\infty_0(\cD,T;\ell_2)$.
In the course of its proof, we use already established results for the regularity of SPDEs in weighted Sobolev spaces with weights that are not powers of the distance to the origin---which for most domains doesn't even make sense---but appropriate powers of the distance 
\[
\rho_\domain(x):=\mathrm{dist}(x,\partial\domain),\qquad x\in\domain,
\]
of a point $x\in\domain$ to the boundary $\partial \domain$ of the underlying domain $\domain\subsetneq\bR^2$.
To avoid confusion, we use the following notation, which became standard in the past decade within the $L_p$-theory for SPDEs on domains. For $p>1$ and $\theta\in\bR$, we write
\begin{align*}
L_{p,\theta}(\domain):=L_p(\domain,\cB(\domain),\rho_{\domain}^{\theta-2}\dx;\bR),
\end{align*}
and 
\begin{align*}
\wsob^1_{p,\theta}(\domain)
:=
\ssggklam{
f: 
\nnrm{f}{\wsob^1_{p,\theta}(\domain)}
:=
\sum_{\abs{\alpha}\leq 1} \nnrm{\rho_\domain^{\abs{\alpha}}D^\alpha f}{L_{p,\theta}(\domain)}<\infty}.
\end{align*}
By $\wsob^{-1}_{p,\theta}(\domain)$ we denote the dual of $\wsob^{1}_{p',\theta'}(\domain)$, where $p'$ and $\theta'$ are related to $p$ and $\theta$, respectively, by the same condition as in~\eqref{eq:dual:wso}.
Similarly, in complete analogy to the $L^{[o]}$- and $\wso$-spaces introduced above, we understand the meaning of 
$L_{p,\theta}(\domain;\ell_2)$, 
as well as the one of 
$\bL_{p,\theta}(\domain,T)$,
$\bL_{p,\theta}(\domain,T;\ell_2)$, 
and
$\bwsob^n_{p,\theta}(\domain,T)=L_p(\ldots;\wsob^n_{p,\theta}(\domain))$ for $n\in\{-1,1\}$.
We refer to \cite[Sections~2.3.3 and 3.1]{Cio2013} and the references mentioned therein for precise definitions and details on these spaces.

\begin{lemma}\label{lem:est:apriori}
Let $p\geq2$, $\theta\in\bR$, 
and let $u$ be a solution in the class $\frwso^1_{p,\theta}(\cD,T)$ of Eq.~\eqref{eq:heatEq:st} with $f\equiv 0$ and $g\in\bwso^\infty_0(\cD,T;\ell_2)$. 
Then
\begin{equation}\label{eq:est:apriori}
\nnrm{u_x}{\bL^{[o]}_{p,\theta}(\cD,T)}
\leq N\sgrklam{
\nnrm{u}{\bL^{[o]}_{p,\theta-p}(\cD,T)}
+ 
\nnrm{g}{\bL^{[o]}_{p,\theta}(\cD,T;\ell_2)}
}
\end{equation}
with a constant $N\in (0,\infty)$ that does not depend on $u$, $g$, or $T$. 
\end{lemma}

\begin{proof}
We first prove that on the stripes
\[
U_1:=\sggklam{x\in\cD : 1<\abs{x}<2}
\qquad
\text{and}
\qquad
V_1:=\sggklam{x\in\cD : 2^{-1}<\abs{x}<4},
\]
we have
\begin{equation}\label{eq:est:apriori:1}
\E\ssgeklam{\int_0^T\!\! \int_{U_1}\Abs{u_x(t,x)}^p\,\dx\,\dt}
\leq
N\,
\E\ssgeklam{\int_0^T \!\! \int_{V_1}\sgrklam{\Abs{u(t,x)}^p+\abs{g(t,x)}_{\ell_2}^p}\,\dx\,\dt},
\end{equation}
with $N\in(0,\infty)$ independent of $T$.
To this end, we localize our equation the following way: Let 
\[
U_1^k:=\ggklam{x\in \cD\colon 2^{-k/4}<\abs{x}<2^{1+k/4}}
\]
for $k=1,2,3$. Furthermore, let $\domain$ be a domain of class $\cont^1_u$ (in the sense of \cite[Assumption~2.1]{Kim2004}) such that
\[
U_1^2 \subseteq \domain \subseteq U_1^3.
\]
Fix a smooth function $\zeta:\cD\to [0,1]$, $\zeta\in \cont^\infty(\overline{\cD})$, such that $\zeta=1$ on $U_1$ and $\mathrm{supp}\,\zeta\subseteq \overline{U_1^1}$. Then
\[
\zeta u \in L_p(\Omega_T;\mathring{W}^1_p(U^1_1)),
\]
see Remark~\ref{rem:zero:bdry}. Thus, in particular, 
\[
\zeta u \in L_p(\Omega_T;\mathring{W}^1_p(\domain)) = \bwsob^1_{p,2-p}(\domain,T);
\] 
the equality holds since $\domain$ is of class $\cont^1_u$, thus a bounded Lipschitz domain and therefore $\mathring{W}^1_p(\domain)=H^1_{p,2-p}(\domain)$, see \cite[Theorem~9.7]{Kuf1980}. Moreover, $u\in\frwso^1_{p,\theta}(\cD,T)$ fulfils
\[
\mathrm{d}u = \Delta u \,\dt + g^k\,\dw^k_t
\]
in the sense of distributions with $u(0,\cdot)=0$, so that, for every $\varphi\in\cont^\infty_0(\domain)$, $\prob$-a.s.,
\[
\grklam{\zeta u(t,\cdot),\varphi}
=
\int_0^t
\grklam{\Delta\grklam{\zeta u(s,\cdot)} + \tilde{f}(s,\cdot),\varphi}\,\dt
+
\sum_{k=1}^\infty
\int_0^t
\grklam{\tilde{g}^k(s,\cdot),\varphi}\,\dw^k_t,
\quad
t\in (0,T],
\]
with
\[
\tilde{f} = u \Delta \zeta - 
2 \rklam{\zeta_x u }_x
\in \bwsob^{-1}_{p,2+p}(\domain,T)
\]
and 
\[
\tilde{g} =\zeta g \in \bL_{p,2}(\domain,T;\ell_2);
\]
here, $\rklam{\zeta_x u}_x:=\rklam{\zeta_{x^{1}}u}_{x^1} + \rklam{\zeta_{x^{2}}u}_{x^2}$.
Therefore, there exists a constant $N\in (0,\infty)$ which, in particular, does not depend on $T$, such that 
\begin{equation}\label{eq:est:apriori:2}
\begin{alignedat}{1}
\nnrm{\zeta u}{\bH^1_{p,2-p}(\domain,T)}
\leq
N
\sgrklam{&
\nnrm{\zeta u}{\bL_{p,2}(\domain,T)}\\
&\,\,\,+
\nnrm{u \Delta \zeta - 
2 \rklam{\zeta_{x}u}_{x}
}{\bH^{-1}_{p,2+p}(\domain,T)}
+
\nnrm{\zeta g}{\bL_{p,2}(\domain,T;\ell_2)}
}.
\end{alignedat}
\end{equation}
This is a consequence of suitable a-priori estimates for the heat equation on domains of class $\cont^1_u$, as they have been derived, e.g., in the course of the proof of \cite[Theo\-rem~2.9]{Kim2004}, see, in particular, Inequality~(5.6) therein and its subsequent consequences on page~282 of \cite{Kim2004}. The independence of the constant on $T$ is not mentioned explicitly in \cite{Kim2004} but it can be deduced by carefully inspecting the derivation of Inequality~(5.6) therein.
The desired estimate~\eqref{eq:est:apriori:1} can be obtained from~\eqref{eq:est:apriori:2} by using 
basic properties of the weighted Sobolev spaces as they can be found, for instance, in \cite{Lot2000}, see also \cite[Section~2.3.3]{Cio2013}. 
In particular, note that
\begin{align*}
\nnrm{u \Delta \zeta - 2 \grklam{\zeta_{x}u}_{x}}{\bH^{-1}_{p,2+p}(\domain,T)}
&\leq
\nnrm{u \Delta \zeta}{\bH^{-1}_{p,2+p}(\domain,T)}
+
2\,\nnrm{ \rklam{\zeta_{x}u}_{x}}{\bH^{-1}_{p,2+p}(\domain,T)} \\
&\leq
N\grklam{\nnrm{u \Delta \zeta}{\bL_{p,2}(\domain,T)}
+
2\,\nnrm{ \zeta_{x}u}{\bL_{p,2}(\domain,T)}}\\
&\leq
N
\nnrm{u }{\bL_{p,2}(\domain,T)},
\end{align*}
since $L_{p,2}(\domain)$ is continuously embedded in $H^{-1}_{p,2+p}(\domain)$ and the (generalized) differentiation operators $v\mapsto v_{x^i}$, $i=1,2$, are bounded from $L_{p,2}(\domain)$ to $H^{-1}_{p,2+p}(\domain)$.
Thus, since $\zeta=1$ on $U_1$, Estimate~\eqref{eq:est:apriori:1} with $N\in (0,\infty)$ independent of $T$ follows from~\eqref{eq:est:apriori:2}.

Now we use Estimate~\eqref{eq:est:apriori:1} to prove the asserted estimate~\eqref{eq:est:apriori} by a dilation argument. For $n\in\bZ$, let
\[
U_n:= \ggklam{x\in\cD : 2^{n-1}< \abs{x}<2^n}
\quad
\text{and}
\quad
V_n:= \ggklam{x\in\cD : 2^{n-2}< \abs{x}<2^{n+1}}.
\]
Furthermore, let $u_n(t,x):=u(2^{2n} t,2^n x)$ and $g_n^k(t,x):=g^k(2^{2n} t,2^n x)$, $k\in\bN$. Then, due to our assumptions on $u$ and $g$, we can verify that $u_n\in\frwso^1_{p,\theta}(\cD,2^{-2n}T)$ and 
\[
\mathrm du_n = \Delta u_n\,\dt + 2^n g^k_n\,\dw^{n,k}_{t}
\]
in the sense of distribution with $u_n(0,\cdot)=0$, 
where $\grklam{w^{n,k}_t}:=\grklam{2^{-n}w^k_{2^{2n}t}}$, $k\in\bN$,  is a sequence of independent one-dimensional Brownian motions.
Thus, we can apply Estimate~\eqref{eq:est:apriori:1} to $u_n$ and after some transformations we obtain
\[
\E\ssgeklam{\int_0^T \int_{U_{n+1}}\Abs{u_x}^p\,\dx\,\dt}
\leq
N\,
\E\ssgeklam{\int_0^T  \int_{V_{n+1}}\sgrklam{2^{-np}\Abs{u}^p+\abs{g}_{\ell_2}^p}\,\dx\,\dt},
\]
Note that the constant $N\in(0,\infty)$ neither depends on $n\in\bZ$ nor on $T$, since the constant in Estimate~\eqref{eq:est:apriori:1} is independent of $T$. The required estimate~\eqref{eq:est:apriori} follows by multiplying the last inequality by $2^{n(\theta-2)}$, using the fact that $\abs{x}\sim 2^n$ on $U_{n+1}$ and $V_{n+1}$, and summing over $n\in\bZ$.
\end{proof}

Finally, using these auxiliary results for equations with smooth and simply structured right hand sides together with our main estimate~\eqref{result stochastic part}, we can prove Theorem~\ref{thm:main:eq} as follows.

\begin{proof}[Proof of Theorem~\ref{thm:main:eq}] As already mentioned in Remark~\ref{rem:uniqueness:det-stoch}, uniqueness is already guaranteed. Thus, we only have to prove existence together with the estimate we claimed. Without loss of generality, we assume that the deterministic forcing term vanishes, i.e., $f\equiv 0$. The general case follows then from what is already known to hold for the deterministic heat equation~\eqref{eq:heatEq:det}, see Remark~\ref{rem:sol+unique:det}.

We first focus on the existence. Fix $g\in\bL^{[o]}_{p,\theta}(\cD,T;\ell_2)$. Then there exists a sequence $\rklam{g_n}_{n\in\bN}\subset\bwso^\infty_0(\cD,T;\ell_2)$ converging to $g$ in $\bL^{[o]}_{p,\theta}(\cD,T;\ell_2)$.
Due to Lemma~\ref{lem:existence:smooth}, for every $n\in\bN$, there exists a unique $u_n\in\frwso^1_{p,\theta}(\cD,T)$ fulfilling
\[
\mathrm{d}u_n = \Delta u_n\,\dt + g^k_n\,\dw^k_t
\]
in the sense of distributions with $u_n(0,\cdot)=0$. Consequently, since we are dealing with linear equations, an application of Lemma~\ref{lem:est:apriori} together with our main estimate from Theorem~\ref{thm stochastic part} yields
\begin{align}\label{eq:est:Cauchy}
\nnrm{u_n-u_m}{\frwso^1_{p,\theta}(\cD,T)}
\leq N\,
\nnrm{g_n-g_m}{\bL^{[o]}_{p,\theta}(\cD,T;\ell_2)},
\qquad n,m\in\bN,
\end{align}
with $N\in(0,\infty)$ independent of $T$. 
Thus, due to the completeness of $\frwso^1_{p,\theta}(\cD,T)$, there exists a unique limit $u$ of $(u_n)_{n\in\bN}$ in $\frwso^1_{p,\theta}(\cD,T)$. 
Its deterministic part is given by $\bD u=\Delta u$ whereas its stochastic part is $\bS u = g$. 
In other words, $u$ is the unique solution in the class $\frwso^1_{p,\theta}(\cD,T)$ of Eq.~\eqref{eq:heatEq:st} with $f\equiv 0$.
Moreover, it is a version of the stochastic convolution
\[
w(t,x)
=
\sum_{k=1}^\infty\int^t_0 \int_{\cD} G(t-s,x,y)g^k(s,y)\,\mathrm dy \,\mathrm dw^k_s,   
\]
since, due to Lemma~\ref{lem:existence:smooth}, $u_n$ is a version of
\[
w_n(t,x)
=
\sum_{k=1}^\infty\int^t_0 \int_{\cD} G(t-s,x,y)g_n^k(s,y)\,\mathrm dy \,\mathrm dw^k_s,   
\]
and, by Theorem~\ref{thm stochastic part},
\[
\nnrm{w_n - w}{\bL^{[o]}_{p,\theta-p}(\cD,T)}
\leq N\,
\nnrm{g_n-g}{\bL^{[o]}_{p,\theta}(\cD,T;\ell_2)},
\]
so that $(u_n)_{n\in\bN}=(w_n)_{n\in\bN}$ converges to $w$ in $\bL^{[o]}_{p,\theta-p}(\cD,T)$. Since, at the same time, $(u_n)_{n\in\bN}$ converges to $u$ in $\bL^{[o]}_{p,\theta-p}(\cD,T)$, we have $u=w$ in $\bL^{[o]}_{p,\theta-p}(\cD,T)$.

Concerning Estimate~\eqref{eq:heatEq:estimate}: For the case $f\equiv 0$, it follows simply by arguing as done above to obtain~\eqref{eq:est:Cauchy}, together with the continuity of the norms. 
The constant used thereby does not depend on $T$.
The general case follows by additionally using Estimate~\eqref{eq:est:Sollonikov}, which holds with a constant independent of $T$ due to~\cite[Theorem~1.2]{Sol2001}.
\end{proof} 


\noindent\textbf{Acknowledgements.} The first author is deeply grateful to the second and to the third author for their amazing hospitality during his stay in Korea. 
The third author conceived and investigated the problem in this article while visiting Germany and expresses  deep thanks to the first and the fourth authors for their hospitality and many useful discussions in Marburg and Kaiserslautern.


\providecommand{\bysame}{\leavevmode\hbox to3em{\hrulefill}\thinspace}
\providecommand{\MR}{\relax\ifhmode\unskip\space\fi MR }
\providecommand{\MRhref}[2]{%
  \href{http://www.ams.org/mathscinet-getitem?mr=#1}{#2}
}
\providecommand{\href}[2]{#2}

\end{document}